\documentclass[10pt,a4paper,twoside]{article}
\usepackage[utf8]{inputenc}

\usepackage{fullpage} 
\usepackage{amsmath, amsthm, amssymb}
\usepackage{thmtools, thm-restate}
\usepackage{mathtools} 
\mathtoolsset{showonlyrefs}
\usepackage{graphicx} 
\usepackage{subcaption} 
\usepackage{dsfont} 
\usepackage{algorithmic} 
\usepackage[norelsize,english,ruled,lined]{algorithm2e} 
\usepackage{color}

\numberwithin{equation}{section}
\newtheorem{theorem}{Theorem}[section]
\newtheorem{lemma}[theorem]{Lemma}

\newtheorem{corollary}[theorem]{Corollary}

\theoremstyle{definition}

\usepackage[colorlinks=true, allcolors=blue]{hyperref}
\hypersetup{
 pdfauthor={H.~Do, J.~Feydy, O.~Mula},
 pdftitle={},
 pdfsubject={},
 pdfkeywords={}
}

\renewcommand{\a}{\ensuremath{\alpha}}
\renewcommand{\b}{\ensuremath{\beta}}
\newcommand{\R}{\mathbb{R}}

\newcommand{\cA}{\ensuremath{\mathcal{A}}}
\newcommand{\cB}{\ensuremath{\mathcal{B}}}
\newcommand{\cC}{\ensuremath{\mathcal{C}}}
\newcommand{\cD}{\ensuremath{\mathcal{D}}}
\newcommand{\cE}{\ensuremath{\mathcal{E}}}

\newcommand{\cL}{\ensuremath{\mathcal{L}}}
\newcommand{\cM}{\ensuremath{\mathcal{M}}}

\newcommand{\cN}{\ensuremath{\mathcal{N}}}

\newcommand{\cP}{\ensuremath{\mathcal{P}}}

\newcommand{\cS}{\ensuremath{\mathcal{S}}}

\newcommand{\cV}{\ensuremath{\mathcal{V}}}

\newcommand{\cX}{\ensuremath{\mathcal{X}}}
\newcommand{\cY}{\ensuremath{\mathcal{Y}}}
\newcommand{\cZ}{\ensuremath{\mathcal{Z}}}

\newcommand{\bE}{\ensuremath{\mathbb{E}}}

\newcommand{\bN}{\ensuremath{\mathbb{N}}}

\newcommand{\bR}{\ensuremath{\mathbb{R}}}

\newcommand{\rA}{\ensuremath{\mathrm{A}}}

\newcommand{\rR}{\ensuremath{\mathrm{R}}}
\newcommand{\rS}{\ensuremath{\mathrm{S}}}

\newcommand{\rX}{\ensuremath{\mathrm{X}}}
\newcommand{\rY}{\ensuremath{\mathrm{Y}}}
\newcommand{\rZ}{\ensuremath{\mathrm{Z}}}

\newcommand{\om}[1]{\textcolor{blue}{#1}} 

\newcommand{\cond}{\ensuremath{\,:\,}}

\newcommand{\Id}{\ensuremath{\text{Id}}}

\newcommand{\charfun}{\ensuremath{\mathds{1}}}

\DeclareMathOperator*{\argmax}{arg\,max}
\DeclareMathOperator*{\argmin}{arg\,min}

\newcommand{\rd}{\ensuremath{\mathrm d}}

\newcommand{\dr}{\ensuremath{\mathrm dr}}

\DeclareMathOperator*{\vspan}{span}

\renewcommand{\Pr}{\ensuremath{\cP_2(\Omega)}}
\newcommand{\bary}{\ensuremath{\mathrm{Bar}}}
\newcommand{\spt}{\ensuremath{\mathrm{spt}}}

\newcommand{\av}{\ensuremath{\text{av}}}
\newcommand{\wc}{\ensuremath{\text{wc}}}

\newcommand{\EE}{\text{EE}}

\newcommand{\PG}{\texttt{PG}}
\newcommand{\AS}{\texttt{AS}}
\newcommand{\GAS}{\texttt{GAS}}
\newcommand{\RGSP}{\texttt{RGSP}}
\newcommand{\BGA}{\texttt{BGA}}
\newcommand{\NN}{\texttt{NN}}
\newcommand{\IDW}{\texttt{IDW}}
\newcommand{\NW}{\texttt{NW}}
\newcommand{\Best}{\texttt{Best}}

\newcommand{\class}{\cA} 
\newcommand{\map}{a}     

\newcommand{\ma}{\alpha} 
\newcommand{\mb}{\beta}  
\newcommand{\mgI}{\nu} 

\newcommand{\barset}{\rA}
\newcommand{\barfun}{\alpha}
\newcommand{\barsetsp}{\rY}
\newcommand{\weight}{\omega} 
\newcommand{\weightcomp}{\textrm{w}} 

\begin{document}
\title{
	Approximation and Structured Prediction \\with Sparse Wasserstein Barycenters 
}
\author{Minh-Hieu Do, Jean Feydy, Olga Mula}
\date{}
\maketitle

\abstract{
	We develop a general theoretical and algorithmic framework for sparse approximation and structured prediction in $\Pr$ with Wasserstein barycenters. The barycenters are sparse in the sense that they are computed from an available dictionary of measures but the approximations only involve a reduced number of atoms. We show that the best reconstruction from the class of sparse barycenters is characterized by a notion of best $n$-term barycenter which we introduce, and which can be understood as a natural extension of the classical concept of best $n$-term approximation in Banach spaces. We show that the best $n$-term barycenter is the minimizer of a highly non-convex, bi-level optimization problem, and we develop algorithmic strategies for practical numerical computation. We next leverage this approximation tool to build interpolation strategies that involve a reduced computational cost, and that can be used for structured prediction, and metamodelling of parametrized families of measures. We illustrate the potential of the method through the specific problem of Model Order Reduction (MOR) of parametrized PDEs. Since our approach is sparse, adaptive and preserves mass by construction, it has potential to overcome known bottlenecks of classical linear methods in hyperbolic conservation laws transporting discontinuities. It also paves the way towards MOR for measure-valued PDE problems such as gradient flows.
    }


\section*{Introduction}

\paragraph{Structured Prediction:} We consider the problem of estimating an unknown function
$f:\cX\to \cY$ from two sets $\cX$ and $\cY$, given a finite number samples $\{ (x_i, y_i)\}_{i=1}^N \subseteq \cX\times \cY$, with $y_i = f(x_i)$. In other words, our goal is to use the finite set of input-output pairs to approximate the full graph of the function $f$, with image
$$
\rY \coloneqq f(\cX) = \{ y = f(x) \cond x \in \cX \} \subset \cY~.
$$
In the following, we identify our training dataset with the three sets:
\begin{equation}
\label{eq:training-sets}
X_N \coloneqq\{x_i\}_{i=1}^N~,\qquad Y_N \coloneqq \{ y_i \}_{i=1}^N~,\qquad S_N=\{(x_i, y_i)\}_{i=1}^N~,
\end{equation}
We will often see these sets as an available dictionary of atoms for our approximation purposes.

The above task, usually known under the name of \emph{least squares fitting}, \emph{regression} or \emph{metamodelling}, arises in numerous applications, and it has a very long history in applied mathematics and computational science. It can be traced back at least to the first efforts in estimating the shape of the earth that took place several millenia ago (see \cite{Nievergelt2000}). Despite this long tradition, least squares problems are still nowadays a topic of very active research. An important modern challenge arises in the case of \emph{structured outputs} that present a rich geometric structure such as as sparsity, a specific graph structure or the membership to a manifold. This structure translates into the output space $\cY$ being nonlinear, and modelling it in mathematical terms is often a difficult task. In this setting, one may be tempted to embed $\cY$ into a larger linear space that may be easier to handle. However, working with the nonlinear space $\cY$ comes with a clear advantage: when suitably incorporated within the learning model, knowledge about the structure of the output space allows one to guarantee the plausibility, and interpretability of the model output. It may also help to decrease the complexity of the problem, both in terms of sample complexity and in terms of model size.

Our contribution fits into this general line of research, which is sometimes called \emph{structured prediction}. We refer to \cite{bakir2007predicting,ciliberto2016consistent,ciliberto2017consistent,osokin2017structured,korba2018structured,rudi2018manifold} for notable contributions in this area, and which have been inspirative for our work. These works cover cases of outputs from spaces of trees, manifolds or label rankings. In this work, we focus on the case where the outputs belong to the space $\cY=\cP_2(\Omega)$ of probability measures defined over a domain $\Omega\subseteq\bR^s$, and with finite second order moments. This specific setting  has only been explored in a few prior works to the best of our knowledge (see \cite{luise2018differential, ELMV2020, BBEELM2022}). It can however be linked with a larger body of contributions that studies how to \emph{approximate}, and \emph{compress} families of measures but where the component of mapping inputs to outputs is not present (see, e.g., \cite{Fletcher2004PGA, Sommer2014, schmitz2018wasserstein, LDCB2021}). As we explain later on, both problems of approximation and least squares fitting are in fact connected because the optimal approximation of a family of measures gives a lower bound on the performance of least squares, structured prediction methods.

Despite the lack of a large body of works on structured prediction in $\cP_2(\Omega)$, the problem appears to be highly relevant given that measures arise as fundamental objects in numerous fields such as economy \cite{carlier2012optimal}, quantum chemistry \cite{cotar2015infinite}, or physical modelling of conservation laws and gradient flows \cite{ambrosio2005gradient}. Probability measures also play a key role in machine learning and imaging. Among the many applications we may cite stand landmark-free shapes and surfaces in computer graphics and medical imaging \cite{glaunes2004diffeomorphic,vaillant2005surface,charon2013varifold,charon2020fidelity}, persistence diagrams in topological data analysis \cite{lacombe2018large,divol2021understanding}, generative modeling \cite{arjovsky2017wasserstein, salimans2018improving}, or predicting cell trajectories \cite{schiebinger2019optimal}. Due to the increasing presence and importance of probability measures in all these fields, studying how to handle them as first-class objects has become a fundamental problem. Our work is a contribution in this direction where the focus lies in understanding how to extend classical least squares strategies from linear vector spaces to the space of measures $\cY=\cP_2(\Omega)$.

\paragraph{Goal: Sparse, adaptive regression in the space of measures:} To estimate the unknown function $f:\cX\to \cY$ from the samples, the main task is to build an approximation map $\map: \cX \rightarrow \cY$ such that, for every input parameter $x\in \cX$, $\map(x)$ approximates $f(x)$ accurately according to some quality criterion. We also ask that the approximation is performed at a reduced computational cost. This is to make the map $a$ be useful in multi-querry contexts arising in applications such as input parameter optimization, or parameter estimation.

In this work, we apply a supervised learning strategy to find a good mapping $a$ by an optimization procedure involving the training samples $\rS_N$. This  requires selecting a priori a model class $\class\subset \cY$ and then to learn the best mapping $a: \cX \to \class$ following some quality criterion. The class translates an educated belief about the geometry/behavior of the set of outputs $\rY = f(\cX)$. In general $\class$ could be either an $n$-dimensional space, or more generally a nonlinear approximation space parametrized by $n$ degrees of freedom, and for every $x$, we have to estimate the parameters to build $a(x) \in \class$. 

The final approximation quality of the mapping will of course dramatically depend on our ability to work with a good class $\class$ but note that designing suitable classes in $\Pr$ is challenging. The lack of linear structure of the space prevents us from working with classical subspaces generated, e.g., by polynomials, radial basis functions or wavelets. Naive neural network parametrizations will also not provide elements in $\Pr$. One could resort to parametrized families of measures such as Gaussian mixtures but in general this type of choice will not allow to interpolate at the training points. In other words, we will not be able to build $a$ such that $a(x)=f(x)$ for $x\in \rS_N$ unless $f(x)$ belongs to the parametrized family.

The most natural class $\class$ that allows for interpolation appears to be the one generated by Wasserstein barycenters from $\rY_N$ as already observed in \cite{luise2018differential, ELMV2020, BBEELM2022}. 
Similarly as in these works, our starting point will be this class, and our development is motivated by the fact that, when the amount of training data $N$ is large, the question of producing a \textit{compressed} barycentric approximation of a given target measure arises. Working in this form is motivated (and even imposed) by several factors:
\begin{itemize}
\item \textbf{Concise representations:} As $N$ increases the dataset may contain functions that are redudant so  sparse representations are expected to distill the most important features in a concise way.
\item \textbf{Stability:} In linear spaces, it is well-known that stability issues arise when working with approximation classes whose dimension is close to the number of observations (see, e.g., \cite{CDL2013}). In our setting, this translates into the fact that the best approximants from the class $\class$ of barycenters with $N$ measures may not be unique. Decreasing the intrinsic dimensionality of the class is therefore expected to mitigate this issue even if uniqueness still cannot be guaranteed. 
\item \textbf{Storage:} State-of-the-art barycenter solvers have a strong memory footprint and it may become even unfeasible to solve the barycenter problem with a too large dataset.
\item \textbf{Numerical Complexity:} Working with sparse barycenters reduces the numerical complexity of solving the barycenter problem, and helps in multi-querry contexts.
\end{itemize}

\paragraph{Contributions of the paper:}
To work with concise barycentric representations, we propose to search for the best approximation of a target measure using only a reduced number $n\leq N$ of samples from the dataset $\rY_N$. This leads us to introduce the notion of best $n$-term barycentric approximation. This concept is interesting in its own right since it is a natural extension in a metric space of the best $n$-term approximation on Hilbert or Banach spaces. For a given target measure, the problem boils down to finding a vector of $N$ barycentric weights which has only $n\leq N$ nonzero entries. We show that this best barycentric approximation can be computed by solving a highly non-convex, bi-level optimization problem. We develop strategies for its practical numerical implementation which deliver satisfactory results in practice as our numerical experiments illustrate.

In the framework of structured prediction, we want to compute, for every input $x\in \rX$, the best $n$-term approximation of $f(x)$ by computing the sparse barycentric weights. However, since this task involves the knowledge of $f(x)$, the approximation with the best $n$-term barycenter is only possible for the available training data points $x$ from $\rX_N$. We thus need to resort to a surrogate strategy which will be suboptimal, but which has to be built in a way to deviate as little as possible from the optimal sparse weights. We propose for this an approach based on Euclidean embeddings in which we learn a metric in the inputs that mimics distances between outputs. This strategie comes with several advantages. First, it is ``model-free'' in the barycentric weights because we do not need to introduce any approximation class to parametrize the behavior of the optimal weights. This point is actually a key novelty with respect to prior contributions where kernel models are used to mimic the behavior of the best barycentric weights (see, e.g., \cite{ciliberto2016consistent, luise2018differential}). Second, the Euclidean embedding approach is invariant to rotations and translations in the parameterization of the space of input vectors. In addition, it is fully adaptive since the $n$ barycentric measures used for the approximation vary with $x$. This is in contrast to prior works such as \cite{ELMV2020, BBEELM2022} where the barycentric measures are selected once and for all.

The theory and methods that we present are general and can be applied to any structured prediction problem with measure-valued outputs. 
Among the possible applications we may mention certain learning tasks in computational anatomy (see, e.g., \cite{pennec2019riemannian, shen2021accurate}), or applications from quantum chemistry where the goal is to learn the probability density of a molecule as a function of inputs such as interatomic positions (see, e.g., \cite{westermayr2022high}). Our main application of interest for the paper has been model order reduction of parametrized PDEs. In the next section, we explain this problem problem more in detail, and we reflect on the relevance of our contribution for that particular field. The application of our ideas to model order reduction is actually another contribution of this work since, as we next explain, it provides a new avenue on how to address certain specific challenges that arise in this field.

Last but not least, we have released on this link
\begin{center}
	\href{https://gitlab.tue.nl/20220022/sinkhorn-rom}{https://gitlab.tue.nl/20220022/sinkhorn-rom}
	\end{center}
the code that we have developped to generate our numerical examples. It contains the implementation of the best $n$-term barycentric approximation for a given known target measure. It also contains our structured prediction approach based on Euclidean embeddings, and we also include the implementation of methods involving heuristic models on the barycentric weights, as well as the non-adaptive greedy barycentric approach initially introduced in \cite{ELMV2020}. 

We hope that the work will spur interest for measure-valued learning in physics and imaging sciences.

\paragraph{Model order reduction:}
At the core of many computational tasks arising in science and engineering is the problem of repeatedly evaluating the output of an expensive physical model for many different instances of input parameter values. Such settings include the numerical solution of parametric Partial Differential Equations (PDEs) for many different values of the parameters, time-stepping for evolutionary PDEs and, more generally, the repeated evaluation of input-output maps defined by black-box computer models. In order to compute the outputs in a reasonable amount of time, it is necessary to find methods that approximate them accurately and at a reduced computational cost.

In the case of parametric PDEs, this task is usually known as \emph{model order reduction} (MOR), \emph{reduced order modeling}, or \emph{metamodeling}. To be concrete, and guide the discussion that follows, consider the prototypical parametric PDE
\begin{equation}
\label{eq:pde}
\cP_x(y)(r) = 0~, \quad \forall r \in \Omega~,
\end{equation}
where $\cP_x$ is a differential operator depending on a parameter $x$, and $y$ is the solution to the PDE (given appropriate boundary conditions). 
Here, and in the rest of the paper, we use $r\in \Omega$ to denote the independent variable in the PDE. $\Omega$ is most often a bounded open set of $\bR^s$. It usually refers to space but it may also refer to more elaborate sets of variables such as time, momentum and possibly other physical variables.

The parameters $x$ often take values in some compact domain $\cX$ of $\bR^d$ but, more generally, $\cX$ can be a compact set of a Banach space. The nature of the solution space $\cY$ strongly depends on the nature of the PDE operator $\cP_x$. In elliptic or parabolic equations, solutions belong to Hilbert spaces. In hyperbolic problems and kinetic problems, solutions are often studied in Banach spaces. There are also numerous relevant PDEs where solutions can be seen as probability distributions. Notable examples are conservation laws, whose solutions are often expressed in $L^1(\Omega)$  but they can also be seen as probability distributions since they preserve mass. Connected to this family stands the broad class of Wasserstein gradient flows which are inherently defined in $\Pr$. To name a few examples on gradient flow problems, we may cite Hele-Shaw flows (see \cite{GO2001}), certain quantum problems (see \cite{GST2009}), porous media flows (see \cite{Otto2001}), Fokker-Planck equations (\cite{JKO1998}), and Keller-Segel models (see \cite{ZM2015, BL2013}). Other examples involving metric spaces that are not necessarily related to gradient flows are the Camassa-Holm equation (\cite{BF2005}), and the Hunter-Saxton equation (\cite{CGH2019}). In addition to this, there are other problems which cannot be defined on Banach spaces, and which can only be defined over metric spaces. Consider for instance the case of a pure transport equation with constant velocity where the initial data is a Dirac measure concentrated on one point. The solution of this PDE remains  at all times a (translated) Dirac mass. More generally, it has been proven that solutions to certain nonlinear dissipative evolution equations with measure-valued initial data are measure-valued and do not belong to some standard Lebesgue or Sobolev spaces. They can however be formulated in the form of Wasserstein gradient flows.

Independently of the nature of the solution space $\cY$, in reduced modeling we are interested in giving fast approximations of the parameter to solution map
$$
f: x \mapsto f(x) = y(x)
$$
from $\cX$ to $\cY$. The structure of interest is that of the compact set of PDE solutions
\begin{equation}
\label{eq:manifold}
\rY \coloneqq \{ y(x) \cond x \in \cX \} \subset \cY,
\end{equation}
which is often referred to as the \emph{solution manifold}, with some abuse of terminology since it may not be a genuine differentiable manifold.

Until the recent works \cite{ELMV2020, BBEELM2022}, research on model order reduction has focused on parametric PDEs posed on Hilbert or Banach spaces. Much attention had been given to elliptic problems where linear approximations such as \cite{HRS2015,QMN2016, BMNP2004,GMNP2007,MM2013,MMT2016, CDS2010,CDS2011, KS2011} provably give very good approximations of the parameter-to-solution mapping. Although the development of nonlinear model reduction methods to tackle broader problem classes is currently very active, the efforts are essentially focused on solutions embedded in Hilbert spaces with Euclidean metric structures (see \cite{} for recent contributions). To the best of our knowledge, methods for measure-valued problems have only been proposed in \cite{ELMV2020, BBEELM2022} despite the large problem classes that fall into this framework. In this landscape, our work can be understood as a contribution in nonlinear model reduction for PDEs posed in $\Pr$ where the metric is non Euclidean. Notably, compared to \cite{ELMV2020, BBEELM2022}, our algorithms are fully adaptive (in a sense that we explain later on) and our implementation allows to treat any spatial dimension. This point is in fact an  important step because the implementation in dimensions larger than one is non-trivial since we cannot rely on closed forms for Wasserstein distances, and barycenters as was leveraged in \cite{ELMV2020, BBEELM2022}.

\paragraph{Structure of the paper:} Section \ref{sec:optimal-approx} defines optimal approximation benchmarks for structured prediction in $\cY = \Pr$. In Section \ref{sec:barycenter-class} we recall the necessary background on Wasserstein spaces and barycenters. This allows us to properly define the class of $n$-sparse barycenters, and the best $n$-term barycentric approximation of a given measure. We show that the best $n$-term approximation is in fact an optimal reconstruction map for structured prediction when working with the class of sparse barycenters. Section \ref{sec:algo-best-n-term} presents numerical algorithms for the practical implementation of the best $n$-term barycentric approximation of a given target measure. Section \ref{sec:interp-weights} presents practical interpolation algorithms to mimic at best the optimal sparse barycentric weights in the framework of structured prediction. After recalling the main existing baseline methods, we present a novel strategy based on fully invariant regression with Local Euclidean Embeddings. Sections \ref{sec:implementation} and \ref{sec:numerical_experiments} are devoted to numerical aspects. All our algorithms require numerous computations of Wasserstein distances, barycenters, and the efficient computation of gradients with respect to distances and barycentric weights. To this end, we rely on a GPU implementation based on entropic regularization that enables fast automatic differentiation. In Section \ref{sec:implementation} we recall the definition of the quantities that are computed by our solvers and provide some specific details about our implementation. Finally, Section \ref{sec:numerical_experiments} illustrates the behavior of all the methods that we discuss in the paper. The study is carried on relatively simple examples from Model Order Reduction of a 2D Burgers' equation.

\section{Approximation benchmarks for structured prediction}
\label{sec:optimal-approx}

\paragraph{Goal:} As explained in the introduction, our goal is to approximate an unknown function $f:\cX\to \cY$ from $N$ data values $\rS_N = \{(x_i, y_i)\}_{i=1}^N$ with $y_i=f(x_i)$. We focus on the case where $\cX$ is a subset of $(\bR^d, \Vert \cdot \Vert)$, and $\cY\subseteq \cP_2(\Omega)$. Using these values, we intend to define a mapping $\map:\cX\to \cY$ such that, for every input parameter $x\in \cX$, $\map(x)$ approximates quickly and accurately $f(x)$. For this, our strategy consists in selecting a priori a model class $\cA \subset \cY$ and we search for the best mapping $a:\cX\to \cA$ according to a quality criterion which we define next.

\paragraph{Approximation Benchmarks:} Assuming that we have fixed the approximation class $\class$, two distinct quality criteria are usually followed to evaluate the performance of a mapping $a:\cX\to \cA$:
\begin{itemize}
\item \emph{Average performance:} We assume that the inputs are distributed following a probability distribution $\rho_\cX$ on $\bR^d$ with support on $\cX$. Given a positive loss function $\cL:\cY\times \cY \to \bR_+$
that behaves as a metric or a divergence on the space of probability distributions $\cY$, the average approximation error with respect to $\rho_\cX$ reads
\begin{equation}
\label{eq:av-loss}
\cE^{\av}(\map) ~\coloneqq~ \bE_{x \sim \rho_\cX} \big[\, \cL \big(\,\map(x),\, f(x)\,\big) \,\big]~.
\end{equation}
\item \emph{Worst case performance:} One can alternatively consider the worst case approximation error
\begin{equation}
\label{eq:wc-loss}
\cE^{\wc}(\map) ~\coloneqq~ \max_{x \in \cX} \cL \big(\,\map(x),\, f(x)\,\big),
\end{equation}
which does not require any assumption on the input distribution. 
\end{itemize}
In both cases, optimal approximation maps are characterized as the set of minimizers of
\begin{equation}
\label{eq:risk-min}
\map^* \in \argmin_{\map: \cX \to \class} \,\cE^{\star}(\map),\quad \star \in \{\av, \wc\}.
\end{equation}
Note that there exists an explicit characterization of an optimal map $a^*$. To see this, for every measure $\ma \in \cY$ let us denote by $P_\cA(\ma)$ the element from $\cA$ that best approximates $\ma$. In other words,
\begin{equation}
\label{eq:best-approx}
P_\cA(\ma) \in \argmin_{\mgI \in \cA} \cL(\mgI, \ma),\quad \forall \ma \in \cY.
\end{equation}
To simplify the presentation, and anticipate practical cases that we consider later on, we will always assume that $\cA$ and $\cL$ are such that the minimum of \eqref{eq:best-approx} is attained (so that we can work with $\min$ instead of $\inf$). The notation in \eqref{eq:best-approx} echoes the fact that $P_\cA : \cY \to \cA$ is a projection operator from $\cY$ onto $\cA$ under the loss function $\cL$.

By construction, we can readily prove that the mapping
\begin{equation}
\label{eq:optimal-map}
a^*(x) = P_\cA f(x), \quad \forall x \in \cX
\end{equation}
is optimal both in the average and in the worst case sense.

\paragraph{Roadmap:} In general we cannot expect to build the optimal map \eqref{eq:optimal-map} using only a finite amount of samples $\rS_N$. Indeed, for every $x\in \cX$, the computation of $P_\cA \left(f(x) \right)$ requires full knowledge of $f(x)$ as \eqref{eq:best-approx} shows but $f(x)$ is only available for the training points $x\in \rS_N$. We thus need to resort to suboptimal mappings $a$. If the loss function $\cL$ satisfies the triangle inequality, the performance of any  map $a:\cX\to \cA$ can be bounded as
$$
\cE^{\star}(\map) \leq \cE^{\star}(\map^*) + \delta^{\star}(\map, \map^*),\quad \star \in \{\av, \wc\},
$$
where
$$
\delta^\av \coloneqq \bE_{x\sim \rho_\cX} \big[\,\cL(a(x), a^*(x))\, \big]\,
\quad \text{and} \quad
\delta^\wc \coloneqq \max_{x\in\cX} \;\cL(a(x), a^*(x)).
$$
This inequality reveals that there are two sources of inaccuracies in suboptimal maps:
\begin{itemize}
\item A \emph{model class error} $\cE^{\star}(\map^*)$ which is connected to the ability of $\cA$ to accurately approximate $\rY=f(\cX)$, the image of $f$.
\item An \emph{algorithm-dependent error} $\delta^{\star}(\map, \map^*)$ which is connected to our ability of mimicking as much as possible an optimal map $a^*:\cX\to \cA$.
\end{itemize}
In our strategy, we will work with a class $\cA$ such that $\rY_N\subset \cA$, and we will build a map $a$ which interpolates the optimal $a^*$ in the sense that
$$
a(x) = a^*(x) = P_\cA f(x) =f(x), \quad \forall x\in \rX_N. 
$$
Our approximation class $\cA$ will be composed of sparse Wasserstein barycenters generated from $\rY_N$ in the sense that we explain in the next section. Also, since our final application focuses on probability distributions involving geometric displacements such as solutions of conservative PDEs, and gradient flows, we choose the Wasserstein-2 metric as the loss function, therefore $\cL=W_2$ in the following (see formula \eqref{eq:W2} in the next section for the exact definition of $W_2$). We emphasize that working with other metrics or divergences in $\Pr$ is also be possible, and the choice should be driven by the final targeted application.

\section{The class of $n$-sparse barycenters}
\label{sec:barycenter-class}

Our work focuses on structured prediction with the class of $n$-sparse barycenters generated from $\rY_N$. To define it, we first recall the necessary background on the Wasserstein metric that is induced by optimal transport theory on a space of probability distributions.

\subsection{Wasserstein space and distance}
\label{sec:wasserstein-space}

Let $\Omega$ be a compact subset in the normed vector space $(\bR^s, \Vert \cdot \Vert)$, with $ \Vert \cdot \Vert$ the Euclidean norm on $\bR^s$. We denote $\cM(\Omega)$ the space of Borel regular measures on $\Omega$ with finite total mass and
\begin{align*}
\cM^+(\Omega) ~&\coloneqq~ \{ \mgI \in \cM(\Omega) \cond \mgI \geq 0 \}~,\\
\cP(\Omega) ~&\coloneqq~ \{ \mgI \in \cM^+(\Omega) \cond \mgI(\Omega)=1 \}~.
\end{align*}
The Wasserstein space $\cP_2(\Omega)$ is defined as the set of probability measures $\mgI\in \cP(\Omega)$ with finite second order moments, namely
$$
\Pr ~\coloneqq~ \{ \mgI \in \cP(\Omega) \cond \int_\Omega \Vert r \Vert^2 \,\rd\mgI(r) \;< +\infty \}~.
$$
We endow the output space $\cP_2(\Omega)$ with the Wasserstein-2 metric that is induced by optimal transport theory which is defined as follows. Let $\ma$ and $\mb$ be two probability measures in $\cP_2(\Omega)$. We define $\Pi(\ma, \mb)\subset \cP_2(\Omega\times \Omega)$ as the subset of probability distributions $\pi$ on $\Omega\times \Omega$ with marginal distributions equal to $\ma$ and $\mb$. The Wasserstein-2 distance between $\ma$ and $\mb$ is defined as:
\begin{equation}
\label{eq:W2}
W_2(\ma,\mb) ~\coloneqq~ \mathop{\inf}_{\pi \in \Pi(\ma,\mb)} \left( \int_{\Omega \times \Omega} \Vert r_0-r_1\Vert^2 \,d\pi(r_0,r_1) \right)^{1/2},\quad \forall (\ma,\mb)\in \cP_2(\Omega) \times \cP_2(\Omega)~.
\end{equation}
As detailed in \cite{Villani2003}, the space $\cP_2(\Omega)$ endowed with the distance $W_2$ is a metric space that is usually called the Wasserstein space.

\subsection{Barycenters}
As already brought up, Wasserstein barycenters arise as natural objects to define a reasonable, and computationally tractable approximation class in $\cY$. Their definition and properties are well understood since their introduction in \cite{AC2011}. Let $N\in \bN^*$ denote the number of observations in our dataset and consider the simplex in $\bR^N$:
$$
\Sigma_N ~\coloneqq~ \Big\{\, \weight_N = (\weightcomp_1,\dots, \weightcomp_N)^T \in \bR^N\cond \weightcomp_i\geq 0,\, \sum_{i=1}^N \weightcomp_i = 1 \, \Big\}~.
$$
Given a set of weights $\weight_N = (\weightcomp_i)_{1\leq i\leq N} \in \Sigma_N$, and given a set $\barset_N = \{\barfun_i\}_{1\leq i\leq N}$ of $N$ probability measures from $\cP_2(\Omega)$, we say that $\bary(\weight_N, \barset_N) \in \cP_2(\Omega)$ is a barycenter associated to $\weight_N$ and $\barset_N$ if and only if:
\begin{equation}
\label{eq:barygen}
\bary(\weight_N, \barset_N)
\in
\arg 
\inf_{\mgI \in \cP_2(\Omega)} \sum_{i=1}^N \weightcomp_i W_2^2(\mgI,\barfun_i)~.
\end{equation}
For subsequent developments, we consider the set of all barycenters $\bary(\weight_N, \barset_N)$ built from $\barset_N$ using weights $\weight_N$ that take values in the simplex $\Sigma_N $:
\begin{equation}
\bary(\Sigma_N, \barset_N) ~\coloneqq~ \{  \bary(\weight_N, \barset_N) \cond \weight_N \in \Sigma_N \} \subset \cP_2(\Omega)~.
\label{eq:bary}
\end{equation}
Assuming that the dataset $\barset_N$ is fixed,
we also introduce a specific notation for the loss function involved in the barycenter problem:
\begin{align}
L~: ~(\weight_N, \mgI) \in  \Sigma_N\times \cP_2(\Omega) ~\mapsto~ L(\weight_N, \mgI) ~\coloneqq~ \sum_{i=1}^N \weightcomp_i W_2^2(\mgI,\barfun_i)\in \bR~. \label{eq:def-h}
\end{align}
Note that $L$ is continuous in the first variable $\weight_N$ and lower semi-continuous in the second variable $\mgI$ with respect to the weak convergence. This is a consequence of the fact that, for any given $\ma\in \Pr$, the map $\mgI\in \cP_2(\Omega) \mapsto W_2^2(\mgI, \ma)$ is lower semi-continuous with respect to the weak convergence.

The existence and uniqueness of minimizers of \eqref{eq:barygen} has been studied in depth in \cite{AC2011}. In particular, this work showed that if one of the distributions $\barfun_i$ in the dataset $\barset_N$ has density with respect to the Lebesgue measure, the barycenter is unique: in the remainder of this paper, we assume that this condition is always satisfied.

For our numerical computations, we must guarantee that for a fixed dataset $\barset_N$, the mapping
\begin{align}
\bar L~:~
\weight_N \in \Sigma_N ~\mapsto~ \bar L(\weight_N) = \min_{\mgI\in \Pr} L(\weight_N, \mgI)\in\bR
\label{eq:bary-proof-compact}
\end{align}
is continuous and differentiable from the simplex $\Sigma_N$ of $(\bR^N, \Vert \cdot \Vert)$ to $\bR$. This is ensured by the following Lemma, whose proof is provided in Appendix \ref{sec:proof-lemma}. We stress that the result does not guarantee any regularity of the arg-minimizers in \eqref{eq:bary-proof-compact}, but only of the value of the minimum.

\begin{lemma}
\label{lemma:cont-weights}
Let $\Omega$ be a compact subset of $(\bR^d, \Vert \cdot \Vert)$ and let $\barset_N =\{\barfun_i\}_{i=1}^N$ be a collection of $N$ measures from $\cP_2(\Omega)$. The application $\weight_N\mapsto \bar L(\weight_N)$ is continuous and differentiable in $\Sigma_N$. In addition, the set $\bary(\Sigma_N, \barset_N)$ of all barycentric combinations of the $\barfun_i$'s is weak sequentially compact in $(\cP_2(\Omega), W_2)$.
\end{lemma}

\subsection{The class $\class_N^n$ of $n$-sparse barycenters}

\paragraph{Motivation and definition:} 
When the number $N$ of measures from our dataset $\barset_N$ becomes large, the question of producing a \textit{compressed} representation of a given measure $\alpha\in \cP_2(\Omega)$ arises. One may thus search for barycentric approximations that involve a reduced number $n\leq N$ of measures from $\barset_N$. As already brought up in the introduction, the construction of such compressed representations is motivated (and even imposed) by several factors:
\begin{itemize}
\item \textbf{Reproducibility:} As $N$ increases, the dataset $\barset_N$ may contain functions that are redudant with respect to the purpose of approximating $\ma$. This translates into the multiplication of local and global minima. On the one hand, this over-parameterization may be advantageous during the numerical optimization; on the other hand, this redundancy may yield too much variability in the final output. Looking for a small number of distributions from $\barset_N$ may promote interpretable interpolation weights that catch the most salient features of the training dataset.
\item \textbf{Storage:} State-of-the-art solvers for the Wasserstein barycenter problem have a memory footprint that is proportional to that of the set of input distributions. In a context where each input distribution may be a high-resolution 3D volume with hundreds of thousands of voxels, keeping the ``batch size'' $N$ small may be a requirement to fit in RAM or GPU memory. As far as reconstruction quality is concerned, interpolating between $32$ high-resolution densities may be more accurate than relying on a full dataset of $1,000$ distributions that must be stored at a lower resolution.
\item \textbf{Numerical complexity:} Since the run time of Wasserstein barycenter solvers scales linearly with the number $N$ of source distributions, finding a subset of the full dataset $\barset_N$ that allows us to best estimate an output $\map(x) \simeq f(x)$ is also desirable to lower computational costs. 
\end{itemize}
In this context, for all size $1\leq n \leq N$, we define the class of $n$-sparse barycenters from $\barset_N$ as
\begin{equation}
\label{eq:n-sparse-bary}
\class_N^n ~\coloneqq~ \bary(\Sigma_N^n, \barset_N) ~=~ \{ \bary(\weight_N, \barset_N) \cond \weight_N \in \Sigma_N^n \}~,
\end{equation}
where
$$
\Sigma_N^n ~\coloneqq~ \{ \weight_N\in \Sigma_N  \cond \vert \spt(\weight_N)\vert \leq n \},
$$
is the set of $n$-sparse vectors from the simplex $\Sigma_N$, and where $\vert \spt(\weight_N)\vert$ denotes the cardinality of the support of $\weight_N$:
$$
\spt(\weight_N) ~\coloneqq~ \{i \in \{1,\dots, N\} \cond \weightcomp_i \neq 0\}~.
$$
It is interesting to remark that $\Sigma_N^n$ is the union of the $n$-simplices embedded in $\bR^N$, and there are $\binom{N}{n}$ of these simplices.

\paragraph{Best $n$-term barycenter:} For all $\ma \in \Pr$, $P_{\cA_N^n}(\ma)$ is the best approximation of $\ma$ with a reduced number $n\leq N$ of barycentric functions from $\barset_N$. We thus call it the best $n$-term barycenter of $\ma$. It can be written as
\begin{equation}
\label{eq:best-n-term-barycenter}
P_{\cA_N^n}(\ma)=
\argmin_{\mgI \in \cA_N^n} W_2(\mgI, \ma) =
 \bary(\weight^n_N(\ma), \barset_N),
\end{equation}
where $\weight^n_N(\ma)$ denotes the optimal $n$-sparse vector of barycentric weights given by
\begin{equation}
\weight^n_N(\ma) ~\in~ \argmin_{\weight^n_N \in \Sigma_N^n} W_2^2(\ma, \bary(\weight_N^n, \barset_N))~.
\label{eq:best-n-term-lambda}
\end{equation}
We stress that since $\weight^n_N(\ma) \in \Sigma_N^n$, it is an $n$-sparse vector in $\bR^N$ that reads
$$
\weight^n_N(\ma) = (0, \dots, 0, \weightcomp_{i_1}, 0, \dots, 0, \weightcomp_{i_n}, 0,\dots)\quad \text{with } \sum_{k=1}^n \weightcomp_{i_k} =1~,\quad 0\leq \weightcomp_{i_k}\leq 1~,\; \forall k\in\{1,\dots,n\}~.
$$
Denoting $I_N^n(\ma) = \{ i_1,\dots, i_n \}$ the set of non-zero indices, we can set
\begin{equation}
\label{eq:adaptive-weights}
\weight_n(\ma) = \{ \weightcomp_i \}_{i \in I_N^n(\ma)} \in \Sigma_n~,
\quad
\barset_N^n(\ma) = \{ \barfun_i \}_{i \in I_N^n(\ma)}
\end{equation}
and equivalently express $P_{\cA^n_N}(\ma)$ as
$$
P_{\cA^n_N}(\ma) = \bary(\weight_n(\ma), \barset_N^n(\ma))~.
$$
This writing emphasizes the fact that the best $n$-term barycenter is defined only with a reduced number $n$ of reference functions from $\barset_N$ and their $n$ associated weights in the simplex of $\bR^n$ (and not $\bR^N$). Also, note that this approximation is \textit{adaptive} since, for each $\ma\in \cY$, we retain the elements from $\barset_N$ that are best suited to this target measure. In section \ref{sec:algo-best-n-term}, we develop numerical algorithms to estimate the best $n$-term barycentric approximation of a given target measure $\ma\in \cY$ by computing the weights $\weight_N^n(\ma)$.

To emphasize the important difference between working with best $n$-term barycenters from $\barset_N$, and working with the best barycenter from a fixed subset $\barset_N^n \subseteq \barset_N$ of size $n\leq N$, we can remark the following. Since $\Sigma_N^n$ contains $\binom{N}{n}$ $n$-simplices, working with $\cA_N^n$ means that we are picking the best approximation among $\binom{N}{n}$ subsets $\barset_N^n$ os size $n$. This is expected to dramatically improve the potential of approximation. In our numerical computations, we have values that are around $N=100$ and $n=10$, thus we are working with about $1.7.10^{14}$ subsets $\barset_N^n$.

\subsection{Structured prediction with $\cA_N^n$ and best $n$-term barycenter}
\label{sec:structured-pred-with-class}

For structured prediction, we work with the class of $n$-sparse barycenters generated by the measures $\rY_N$ from the training set. In other words, we set $\barset_N = \barsetsp_N$, and we work with
$$
\cA_N^n = \bary(\Sigma_N^n, \rY_N).
$$

We can now apply the general framework on optimal approximation maps from section \ref{sec:optimal-approx} for our selected class. From \eqref{eq:optimal-map} and \eqref{eq:best-n-term-barycenter}, it follows that an optimal reconstruction map is
$$
a^*(x)
= P_{\cA_N^n} \left( f(x) \right)
= \argmin_{\mgI \in \cA_N^n} W_2(\mgI, f(x))
= \bary(\weight^n_N(f(x)), \barsetsp_N)
, \quad \forall x\in \cX,
$$
where the weights satisfy
\begin{equation}
\weight^n_N(f(x)) ~\in~ \argmin_{\weight^n_N \in \Sigma_N^n} W_2^2(f(x), \bary(\weight_N^n, \barsetsp_N))~.
\label{eq:best-n-term-lambda-fx}
\end{equation}
The above equations shows that computing the weights $\weight^n_N(f(x))$ from the optimal map $a^*(x)= P_{\cA_N^n} \left( f(x) \right)$ inevitably requires the full knowledge of $f(x)$. Except for the input $x\in \rX_N$, this computation cannot be performed because we only have access to the input $x$ at runtime. We thus need to find a strategy to build sparse weights $\widehat \omega_N^n(x)$ such that
\begin{equation}
\label{eq:approx-weights}
\widehat \omega_N^n(x) \approx \omega_N^n( f(x) ), \quad \forall x\in \cX.
\end{equation}
In Section \ref{sec:interp-weights}, we develop an adaptive interpolation strategy to build $\widehat \omega_N^n(x)$. It is based on the construction of a Euclidean embedding in which we learn a metric in the inputs that mimics distances between the outputs. The metric is invariant to reparametrizations of the input such as rotations, translations and dilations. We claim that the approach is interpolatory because our construction will be such that
\begin{equation}
	\widehat \omega_N^n(x_i) = \omega_N^n( f(x_i) ) = e_i, \quad \forall i\in \{1,\dots, N\},
\end{equation}
where $e_i$ is the $i$-th unit vector of $\bR^N$. This will ensure that our final approximation $a$ satisfies
$$
a(x) = f(x), \quad \forall x\in \rX_N.
$$

\subsection{Additional comments}
Before moving to the next section, several comments are in order:
\paragraph{Sources of Error:} In Section \ref{sec:optimal-approx}, we explained that any approximation map $a:\cX\to \cA$ will inevitably suffer from two error sources. When working with the class $\cA_N^n$ from \eqref{eq:n-sparse-bary}, they can be understood as follows:
\begin{itemize}
	\item The model error is connected to the quality of the class $\cA_N^n$ to approximate $f(\cX)$. As $N$ increases, the approximation quality of $\cA_N^n$ improves. In fact, $\cE^\star(a^*)$ tends to 0 as $N\to \infty$ but it has so far not been possible to build a theory on the convergence rate of this quantity as a function of $N$. This point is left as an open problem. Also, as we explain in the next paragraph, if we fix $N$ and we increase $n$, the quality of approximation will increase but we will lose in sparsity.
 \item The algorithmic error is linked to our ability to approximate the best sparse weights as expressed in \eqref{eq:approx-weights}.
\end{itemize}

\paragraph{Performance hierarchy:}
Denoting 
$$
\sigma(\ma, \class_N^n) ~\coloneqq~ \min_{\mgI \in \class_N^n} W_2(\ma, \mgI)~.
$$
the best approximation error of $\alpha$ with $\cA_N^n$, and using that the sequences are nested in the sense that
$$
\class_N^n \subseteq \class_N^m \subseteq \class_N^N = \bary(\Sigma_N, \barset_N), \quad 1\leq n\leq m \leq N~,
$$
we derive the hierarchy of approximation errors
\begin{equation}
\label{eq:nested-approx}
\sigma(\ma, \class_N^n)
\geq \sigma(\ma, \class_N^m)
\geq \sigma(\ma, \class_N^N), \quad 1\leq n\leq m \leq N~.
\end{equation}
These inequalities higlight the fact that, from an approximation perspective, the set of $n$-sparse barycenters from $\class_N^n=\bary(\Sigma_N^n, \barset_N)$ is less desirable than the full $N$-dimensional polytope $\class_N^N=\bary(\Sigma_N, \barset_N)$. However, for the reasons outlined above, trade-offs between accuracy and interpretability or numerical performance motivate the use of sparse barycenters when $N\gg1$.

\paragraph{Why do we call $P_{\cA_N^n}(\alpha)$ the best $n$-term barycenter of $\alpha$?} The terminology echoes the well-defined notion of best $n$-term approximation on a Hilbert space $(V, \Vert \cdot \Vert_V)$. It expresses the best approximation that one can achieve when approximating a function from a redundant dictionary $\cD \subset V$. Among the usual properties required of the dictionary stand that it is complete, namely that $V = \overline{\vspan{\cD}}$. For a given $\alpha\in V$, one is then interested in expressing $\alpha$ as a linear combination of at most $n$ elements:
\begin{equation}
\label{eq:sn}
\ma_n = \sum_{g \in \cD} \weight_g g ~,\quad \#\spt(\weight)\leq n~.
\end{equation}
The best approximation error in this case is thus
$$
\sigma(\ma, \cD) \coloneqq \inf_{\ma_n} \Vert \ma - \ma_n \Vert
$$
where the approximations $\ma_n$ are of the form \eqref{eq:sn}.

In our setting, the set of measures from the dataset $\rY_N$ plays the role of the dictionary $\cD$. Unlike $\cD$, our dataset $\rY_N$ is not assumed to be complete nor dense in $\Pr$.

\section{Numerical algorithms for best $n$-term approximation}
\label{sec:algo-best-n-term}

In this section, we present algorithms that approximate the optimal barycentric weights $\weight_N^n(\ma)$ from problem \eqref{eq:best-n-term-lambda} which we recall here:
\begin{equation}
\weight^n_N(\ma) ~\in~ \argmin_{\weight^n_N \in \Sigma_N^n} W_2^2(\ma, \bary(\weight_N^n, \barset_N))~.
\label{eq:best-n-term-lambda-2}
\end{equation}
A salient feature of the proposed algorithms is that they always involve computations of barycenters with a reduced number $n$ of components, therefore avoiding the expensive computation of barycenters with $N\gg1$ components. As our numerical experiments will illustrate, the algorithms produce good approximations of the optimal sparse weights. Also, as explained in Section \ref{sec:structured-pred-with-class}, $\weight^n_N(\ma)$ cannot be computed in the framework of structured prediction when $\alpha = f(x)$. However, deriving an algorithm to solve the best $n$-term approximation is of interest in its own right, and it can be used for benchmarking purposes for structured prediction. 

Going beyond the fact that the minimizer of \eqref{eq:best-n-term-lambda-2} is not unique, the optimization problem is challenging due to the non-convexity of the objective function. We also note that the constraint that our vector of weights $\weight_N^n$ belongs to $\Sigma_N^n$ is not trivial to handle (note that $\Sigma_N^n$ is a union of $n$-simplices embedded in $\bR^N$). A tempting strategy could be to approximate this requirement by the sum of a hard constraint that $\weight_N^n$ belongs to the $N$-simplex $\Sigma_N$ and a soft $\ell_1$-relaxation to promote sparsity. However, this idea would not work because for any $\weight_N\in \Sigma_N$, we have $\Vert \weight_N \Vert_{\ell_1(\bR^N)} = 1$ independently of the sparsity pattern of $\weight_N$. Since standard $\ell_1$ penalization cannot enforce sparsity in our case, we opt for a direct approach based on a projected gradient descent onto the set $\Sigma_N^n$, implemented using the Greedy Selector and Simplex Projector algorithm (GSSP) from \cite{KBCK2013}.

Our first possible algorithmic variant, summarized in Algorithm \ref{alg:GraFS}, fixes a sparsity degree $n$ and performs a simple gradient descent step followed by a projection onto $\Sigma_N^n$. In this algorithm,
\begin{align}
P_{\Sigma_N^n} : \bR^N &\to \bR^N \\
v &\mapsto P_{\Sigma_N^n} v \coloneqq \arg \min_{z \in \Sigma_N^n} \Vert v - z \Vert^2_2
\end{align}
denotes the (Euclidean) orthogonal projection into  $\Sigma_N^n$ (computed with GSSP in practice). We call this variant \PG~for \underline{P}rojected \underline{G}radient.

We formulate several variants of Algorithm~\ref{alg:GraFS}. Instead of fixing the parameter $n$ for all iterations, one can set this parameter equal to the length of the support at each iteration: this corresponds to Algorithm~\ref{alg:GAS}, which we call \GAS~for \underline{G}radient descent with \underline{A}daptive \underline{S}upport. Moreover, motivated by the Gradient Support Pursuit proposed in \cite{bahmani2013greedy} in the context of compressed sensing, we present a \underline{R}estricted \underline{G}radient \underline{S}upport \underline{P}ursuit (\RGSP) strategy in Algorithm~\ref{alg:RGSP}. In that algorithm, the first step in each iteration is to evaluate the gradient $z = \nabla_\weight \cL(\ma, \bary(\weight, \barset_N)) \in \bR^N$ of the cost function at the current estimate $\omega$. Then we pick the $2n$ coordinates of $z$ that have the largest magnitude. These coordinates are chosen as the directions in which pursuing the minimization is the most effective. Their indices, denoted by $\cZ$, are then merged with the support of the current estimate to obtain the set of indices $\cS$. This is a set of at most $3n$ indices over which the loss function is then minimized to produce an intermediate estimate $\tilde \omega \in \Sigma_N^{3n}$. Note that this step will typically involve computing barycenters with at most $3n$ components, and  derivatives of at most $3n$ direction components. Finally, we project $\tilde\omega$ onto $\Sigma_N^n$.  


\begin{algorithm}[htp] 
	\SetKwInput{KwData}{Input}
	\SetKwInput{KwResult}{Output}
	\SetCustomAlgoRuledWidth{2cm}
	 \KwData{Target measure $\ma$, reference set $\barset_N$, sparsity degree $n$, learning rate $\tau>0$.}
	 \KwResult{Approximation of $\weight_N^n(\ma)$}
	 Initialize $\weight \in \Sigma^n_N$\;
	 \Repeat{convergence}{
	$\weight \gets P_{\Sigma_N^n} \left( \weight - \tau \nabla_\weight \cL(\ma, \bary(\weight, \barset_N)) \right)$
	}
	 \caption{\underline{P}rojected \underline{G}radient descent with fixed support $n$ (\PG)}
	 \label{alg:GraFS}
\end{algorithm}

\begin{algorithm}[htp]
	\SetKwInput{KwData}{Input}
	\SetKwInput{KwResult}{Output}
	\SetCustomAlgoRuledWidth{2cm}
	\KwData{Target measure $\ma$, reference set $\barset_N$, learning rate $\tau>0$, sparsity threshold $n_{\max}$.}
	\KwResult{Approximation of $\weight_N(\ma)$ with sparsity degree $n\leq n_{\max}$ obtained dynamically.}
	Initialize $\weight \in \Sigma^{n_{\max}}_N$\;
	\Repeat{convergence}{
	\begin{equation*}
		\widetilde \weight \gets  \weight - \tau \nabla_\weight\, \cL(\ma, \bary(\weight, \barset_N))
	\end{equation*}
	\begin{equation*}
		n \gets \min(n_{\max}, \vert \spt ( \widetilde\weight)\vert)
	\end{equation*}
	\begin{equation*}
		\weight \gets P_{\Sigma_N^n}(\widetilde{\weight})
	\end{equation*}
	
	}
	\caption{\underline{G}radient descent with \underline{A}daptive \underline{S}upport (\GAS)}
	\label{alg:GAS}
\end{algorithm}

\begin{algorithm}[htp]
	\SetKwInput{KwData}{Input}
	\SetKwInput{KwResult}{Output}
	\SetCustomAlgoRuledWidth{2cm}
	\KwData{Target measure $\ma$, reference set $\barset_N$, sparsity degree $n$, learning rate $\tau>0$.}
	\KwResult{Approximation of $\weight_N(\ma)$ with sparsity degree obtained dynamically.}
	Initialize $\weight \in \Sigma^{n}_N$ \;
	\Repeat{convergence}{
		Compute the gradient:
		$z = \nabla_\weight \cL(\ma, \bary(\weight, \barset_N)) \in \bR^N$ 
		
		Identify the indices of the $2n$ largest components of $z$:
		$\cZ = \spt \left( z_{2n}\right) $ 
		
		Merge supports:
		$ \cS = \cZ \cup \spt \left(\weight\right)$

		Solve
		$$
			\tilde \omega \gets \argmin_{ \substack{\omega \in \Sigma_N \\ \omega |_{\cS^c} = 0 } } \cL(\ma, \bary(\weight, \barset_N))
		$$

		
	 	Projection onto the sparse simplex: $\omega = P_{\Sigma_N^n} (\tilde \omega)$
		
	}
	\caption{\underline{R}estricted \underline{Gra}dient \underline{S}upport \underline{P}ursuit (\RGSP)}
	\label{alg:RGSP}
\end{algorithm}

\newpage

\section{Effective interpolation algorithms for structured prediction}
\label{sec:interp-weights}
In the framework of structured prediction, given an input $x\in \cX$ we cannot approximate $f(x)$ with the best $n$-term barycenter $P_{\cA_N^n}f(x)$ since the computation of the weights $\weight_N^n(f(x))$ requires the full knowledge of $f(x)$, and this information is only available for the inputs $x$ from the training set $\rX_N$. Our main challenge is to extend this information to $x$ in a robust and reliable way. We now present several strategies to compute such a vector of surrogate weights before performing an experimental evaluation in Section~\ref{sec:numerical_experiments}. In Section \ref{sec:other-interp}, we recall the two main existing approaches from the literature. Section \ref{sec:LEE} describes a novel approach that we propose as an improvement.

\subsection{Baseline methods}
\label{sec:other-interp}
\paragraph{Truncated Nadaraya--Watson kernel interpolation:} 
A classical strategy to define barycentric interpolation weights is to rely on a positive kernel function $K:\cX\times \cX\to \bR_+$. Common choices include the Gaussian kernel with deviation $\sigma > 0$:
\begin{equation}
K(x,\tilde x)~=~\exp(-\|x-\tilde x\|^2_{\ell_2(\bR^d)} / 2\sigma^2)
\label{eq:K-NS}
\end{equation}
and the inverse distance weight with exponent $p\geq 1$ and regularization $\eta > 0$:
\begin{equation}
K(x,\tilde x) ~=~ (\eta+\Vert x - \tilde x\Vert_{\ell_2(\bR^d)})^{-p}~.
\label{eq:K-IDW}
\end{equation}
In order to create an $n$-sparse vector of barycentric interpolation weights $\weight_N^n(x)$ for a given input parameter $x\in X$, we then proceed in two steps:
\begin{enumerate}
	\item Compute the set $\cN_n(x)\subset [1,N]$ of indices that correspond to the $n$-nearest neighbors of $x$ in the dataset $X_N=\{x_1, \dots, x_N\}$, for the Euclidean distance $\ell_2(\bR^d)$ on the space of parameters. 
	\item Create a $n$-sparse vector $\widehat \weight_N^n(x)$ whose $i$-th coordinate is equal to $0$ if $i$ does not belong to the set of neighbors $\cN_n(x)$, and is otherwise equal to:
	\begin{equation}
	\label{eq:nadaraya}
	(\widehat \weight_N^n(x))_i \coloneqq \frac{ K(x, x_i) }{\sum_{j\in \cN_n(x)} K(x, x_j)}\; \in [0, 1]~.
	\end{equation}
\end{enumerate}

Implementing this heuristic is straightforward, but raises the difficult question of the choice of a kernel function $K$ on the space of input parameters. Moreover, since this method relies on the Euclidean metric on the space of parameters $\cX$ instead of the Wasserstein metric on the space of output distributions $\cY$, it is not invariant to the parameterization of the input space $\cX$. This is problematic for applications to model order reduction, where vectors of input parameters $x \in\cX$ usually correspond to physical parameters whose units and scalings may vary.

We note that as a refinement of this approach, one may consider using a kernel ridge regression method to derive the barycentric weights. An issue that often arises is that the positivity of the weights is no longer ensured and one needs to resort to thresholding in practice. Overall, the resulting algorithm has similar merits and limitations to the use of \eqref{eq:nadaraya}. We omit a detailed presentation from the main text of the paper and refer the interested reader to \cite{CRR2016, CRR2020}.

\paragraph{Barycentric Greedy Algorithm (\BGA), a sparse but non-adaptive approach:}
One of the main challenges of interpolation methods based on Wasserstein barycenters is the choice of the support of the vector of weights $\weight_N^n$. As discussed above, relying on the set of $n$-nearest neighbors of a sampling point $x\in\cX$ for the Euclidean metric may be problematic. In order to bypass the  questionable choice of a parameterization of the space of input parameters $\cX$, the authors of \cite{ELMV2020, BBEELM2022} propose an optimality criterion that only makes use of the Wasserstein-2 distance on the output space $\cY = \cP(\Omega)$. These two methods present a greedy algorithm that selects a \emph{fixed} subset $\rY_n$ of $n$ measures that is most representative of the training set $\rY_N$:
\begin{enumerate}
	\item Initialization (n=2): Find a pair of input parameters $(x_1, x_2)\in X_N \times X_N$ in the dataset with:
	$$
	(x_1,x_2) \in \argmax_{(x, \tilde x)\in X_N\times X_N} W_2(f(x), f(\tilde x))~.
	$$
	Use the two extremal distributions of our dataset as reference measures by setting $\rY_2 \coloneqq \{ f(x_1), f(x_2) \}=\{y_1,y_2\}$.
	\item Induction step: For $n\geq 3$, assume that we have computed $\rY_{n-1} = \{y_1,\dots, y_{n-1}\}$. Search for a data sample $y_n = f(x_n)$ that is as far as possible from the Wasserstein polytope induced by $\rY_{n-1}$, i.e. choose:
	\begin{equation}
	x_n \in \argmax_{x\in X_N} \min_{\weight_{n-1}\in \Sigma_{n-1}} W_2^2(f(x), \bary(\weight_{n-1}, \rY_{n-1}))
	\label{eq:greedy-induction}
	\end{equation}
	and set $\rY_n = \rY_{n-1} \cup \{ y_n \}$ with $y_n = f(x_n)$.
\end{enumerate}

This generalized farthest point sampling produces a fixed set $\rY_n$ whose Wasserstein convex hull is most representative of the full training set $\rY_N$.
Then, for every new target $x\in X$, the authors of \cite{ELMV2020, BBEELM2022} propose to use as output distribution the barycenter:
$$
\map(x) = \bary(\widetilde \weight_n(x), \rY_n)~\simeq~ f(x)~,
$$
where the vector of $n$ weights $\widetilde\weight_n(x)\in \Sigma_n$ of dimension $n$ is built with an interpolation procedure. The core idea of \cite{ELMV2020, BBEELM2022} is to compute optimal barycentric weights $\weight_n(x_i)$ in the $n$-simplex $\rY_n$ for the training distributions $\rY_N=\{f(x_1), \dots, f(x_N)\}$, and then build an interpolating function  $x\mapsto \widetilde \weight_n(x)$ such that
$\weight(x_i)=\widetilde \weight(x_i)$ for $i=1,\dots, N$.

\subsection{Fully invariant regression with Local Euclidean Embeddings}
\label{sec:LEE}

\paragraph{Adaptive, sparse algorithms:} 
Both of the baseline methods discussed above have clear limitations. On the one hand, kernel-based interpolations rely on the choice of a relevant parameterization of the space of input parameters $x$ and of a suitable kernel function with appropriate choice of hyperparameters. This requires some expertise from end-users and may be a barrier to adoption. On the other hand, the \BGA~method of \cite{ELMV2020, BBEELM2022} selects reference measures using an interpretable geometric criterion on the space of output distributions $y$, but may be unable to model complex data distributions with a single Wasserstein simplex of dimension $n$.
In order to define a new parameter-free method for regression in Wasserstein space, we propose to rely instead on \emph{adaptive} Wasserstein $n$-simplices and on a locally Euclidean (Riemannian) metric on parameter space $\cX$ that approximates the \emph{pull-back} of the Wasserstein-2 metric from $\cY$ by the application $f:\cX\rightarrow \cY$.

\paragraph{Metric approximation of the Wasserstein projection problem:}
First, let us consider a given parameter $x\in X$ and the  optimization problem:
\begin{equation}
\label{eq:weights-EE-v1}
\min_{\weight_N^n \in \Sigma_N^n}
\sum_{i=1}^N \big\vert\, W_2^2\big(f(x), f(x_i)\big) ~-~ W_2^2\big(\bary(\weight_N^n, \barsetsp_N), f(x_i)\big)  \,\big\vert~.
\end{equation}
This problem is similar to the projection of $f(x)$ on $\cA_N^n$ that we presented in \eqref{eq:best-n-term-lambda-fx}: we look for a barycenter $\bary(\weight_N^n, \barsetsp_N)$ whose distances to the data samples $f(x_i)$ are as close as possible to the ground truth distances $W_2\big(f(x), f(x_i)\big)$.

The main advantage of this metric formulation is that even if $f(x)$ is unknown, we may still approximate the terms related to $f(x)$ with local quadratic functions that will act as local Euclidean embeddings:
$$
W^2_2(f(x), f(x_i)) \approx (x-x_i)^T M(x_i) (x-x_i)~,\quad \forall i\in \{1,\dots, N\}~.
$$
In the equation above, $M(x_i)$ is a positive definite matrix of size $d\times d$.

\paragraph{Computation of the Local Euclidean Embeddings:} 
The collection of local Euclidean metrics $M(x_i)$ defines a data-driven approximation of the pull-back metric of the Wasserstein-2 distance by the application $f:\cX\rightarrow\cY$. We compute this field of $d\times d$ matrices during an offline training phase. First, we obtain the $N\times N$ matrix of pairwise Wasserstein-2 distances between the elements $y_i = f(x_i)$ of our training set $\rY_N$:
$$
D = ( d_{i,j} )_{1\leq i, j \leq N}~, \quad d_{i,j} \coloneqq W_2(y_i, y_j)~.
$$
Then, for every training data pair $(x_i, y_i)\in \rS_N$, we use CvxPy \cite{diamond2016cvxpy,agrawal2018rewriting} to solve:
\begin{equation*}
	M^*(x_i) \in \argmin_{M \succcurlyeq 0} \sum_{j=1}^N \big(\, (x_i-x_j)^T (M+\eta \Id) (x_i-x_j) ~-~ d_{i,j}^2 \,\big)^2~,
\end{equation*}
and use as local Euclidean metric $M(x_i)$ the solution:
\begin{equation}
\label{eq:learned-EE}
M(x_i) = M^*(x_i) + \eta \Id~.
\end{equation}
In the above formulas, $\Id$ is the $d\times d$ identity matrix and $\eta>0$ is a regularization parameter. The latter is introduced to mitigate conditionning issues and guarantee that the resulting matrix $M(x_i)$ has full rank.

\paragraph{Adaptive projection:}
This approximation yields the tractable optimization problem for the interpolation weights:
\begin{equation}
\label{eq:weights-EE-v2-full}
\weight^{n,\EE}_N(x) \in 
\min_{\weight_N^n \in \Sigma_N^n}
\sum_{i=1}^N \big\vert\, (x-x_i)^T M(x_i) (x-x_i) ~-~ W_2^2\big(\bary(\weight_N^n, \barsetsp_N), f(x_i)\big) \,\big\vert~,
\end{equation}
where the superscript $\EE$ stands for ``Euclidean Embedding''. Following Section \ref{sec:algo-best-n-term}, we can solve this optimization problem using a projected gradient descent. In order to speed up computations for large values of $N$ and reduce the influence of points $x_i$ that are too far away from $x$, we propose to truncate the sum above to $k$ terms with $n\leq k\leq N$. To make our interpolation process local without relying on the parameterization of the input space $\cX$, we use the $k$ points $x_i$ that are associated to the smallest values of $(x-x_i)^T M(x_i)(x-x_i) \approx W_2^2(f(x), f(x_i))$.

\paragraph{Motivations and properties:}

Our data-driven approach to compute the barycentric weights $\weight_N^{n, \EE}(x)$ has the following properties:
\begin{itemize}
	\item \textbf{Interpolation:} If $x=x_i$ for some $x_i$ in the training set, then $\weight_N(x)=e_i$ is a solution to \eqref{eq:weights-EE-v1}. We perfectly recover the target image since $\bary(\weight_N(x), \barsetsp_N) = \bary(e_i, \barsetsp_N) = y_i$.
	\item \textbf{Sparsity:} We never use our Wasserstein barycenter solver with more than $n$ data distributions at a time.
	\item \textbf{Full adaptivity:} The support of the vector of weights $\weight^{n, \EE}_N(x)$ is optimized with respect to $x$. 
	Instead of relying on a unique simplex $\barsetsp_n$ as in \cite{ELMV2020, BBEELM2022}, our regression method may use all possible $n$-simplices from $\barsetsp_N$ for different values of the input vector $x$.
	\item \textbf{Robustness to the input space $\cX$:} Thanks to the family of local Euclidean metrics $M(x_i)$, our method is robust to changes of coordinates for the input vectors $x$.
\end{itemize} 

We note that our metric re-formulation of the Wasserstein projection problem in \eqref{eq:weights-EE-v1} has some reproducing properties. To discuss them, we start by recalling the following result, which is crucially used in Karl Menger's works on characterizing metric spaces that are isometrically embedabble in finite dimensional Euclidean spaces (see \cite{BB2017}). 
\begin{lemma}
\label{lem:cayley}
Let $\cZ$ be a $q$-dimensional Euclidean space with inner product $\left< z, \tilde z  \right>$, norm $\Vert z \Vert = \sqrt{\left< z, z  \right>}$, and  distance $d_\cZ (z, \tilde z) = \Vert z - \tilde z\Vert$. Let $\rZ_k =\{ z_1,\dots, z_k\}$ be a set of $k\leq q$ linearly independent vectors of $\cZ$, and let
$$
\bary(\Sigma_k, \rZ_k) = \{ z = \sum_{i=1}^k \weightcomp_i z_i \cond \weightcomp\in \Sigma_k  \} \subset \cZ
$$
be the set of barycenters associated to $\rZ_k$. Let $a \in \bary(\Sigma_k, \rZ_k)$ and let $b\in \cZ$. If $d_\cZ(a, z) = d_\cZ(b, z)$ for all $z\in \rZ_k$, then $a=b$.
\end{lemma}
In other words: in Euclidean spaces, distances to the vertices of a simplex uniquely determine locations on the simplex.
As a consequence of this result, we observe that our construction from \eqref{eq:weights-EE-v1} guarantees perfect reconstruction in some favorable cases as we record in the following corollary.
\begin{corollary}
\label{cor:diracs}
Let $\{x_i\}_{i=1}^k$ be a set of distinct points in $\bR^d$, and let $\rY_k = \{ \delta_{x_i} \}_{i=1}^k\subset \cY$. Let $\alpha \in \bary(\Sigma_k, \rY_k)$ and let $\beta=\delta_r \in \cY$ for $x\in \bR^d$. If $W_2(\alpha, \delta_{x_i}) = W_2(\beta, \delta_{x_i})$ for all $i\in\{1,\dots, k\}$, then $\alpha=\beta$. The same result holds true for a family of gaussians $\widetilde \rZ_k = \{ \cN(x_i, C) \}_{i=1}^k\subset \cY$ with constant covariance $C$.
\end{corollary}
\begin{proof}
The proof is based on the fact that there is an isometric isomorphism between the set of Dirac masses $\{\delta_x\}_{x\in \Omega}$ in $(\cP_2(\Omega), W_2)$ with $(\bR^d, \Vert\cdot\Vert_{\ell_2(\bR^d)})$, and we have that $W_2(\delta_x, \delta_{\tilde x}) = \Vert x - \tilde x\Vert_{\ell_2(\bR^d)}$. The same holds true for $\{\cN(x, C)\}_{x\in \bR^d}$, the set of gaussians with constant covariance $C$. The proof then follows by applying Lemma \ref{lem:cayley} to the points $\{x_i\}_{i=1}^k$ in $(\bR^d, \Vert\cdot\Vert_{\ell_2(\bR^d)})$.
\end{proof}
Corollary \ref{cor:diracs} ensures basic reproducing properties of problem \eqref{eq:weights-EE-v1}. It implies that, in the very simple case where the function $f$ is defined as
\begin{align}
f: X &\to \cY \nonumber\\
x &\mapsto f(x) = \delta_x \text{ (or $\cN(x, C)$)} 
\end{align}
then $\barset_N$ is of the form $\barset_N =\{\delta_{x_i}\}_{i=1}^N$, and problem \eqref{cor:diracs} simply reads
\begin{equation}
\min_{\weight_N \in \Sigma_N} \sum_{i=1}^N \vert \Vert x - x_i \Vert^2_{\ell_2(\bR^d)} - \Vert \sum_{j=1}^N \weightcomp_j (x_j - x_i) \Vert^2_{\ell_2(\bR^d)} \vert^2 .
\label{eq:weights-EE-v1-diracs}
\end{equation}
This has the following implication. Suppose that $x$ is a given input and we have to approximate the target output $y = f(x)=\delta_x$. Suppose that $y$ is in the convex hull of $\barset_N$, namely there are weights $\weight_N(x)\in \Sigma_N$ such that $x = \sum_{i=1}^N \weightcomp_i(x) x_i$. Then corollary \ref{cor:diracs} guarantees the existence of a unique minimum in \eqref{eq:weights-EE-v1-diracs} which is attained at the best barycentric weights  $\weight_N(x)$. In other words, problem \eqref{eq:weights-EE-v1} recovers the exact weights in the case of Dirac masses. Note in addition that in this simple example there is an exact Euclidean embedding which is defined by taking the Gramm matrices $M(x)=\Id$. Therefore formulations \eqref{eq:weights-EE-v1} and \eqref{eq:weights-EE-v2-full} are exactly equivalent in this case (no approximation is added by the Eucliden embedding).

We can slightly generalize the above example an consider mappings $f(x) = \delta_{g(x)}$ or $f(x) = \cN(g(x), C)$, where $g(x)\coloneqq Ax +b$ is an affine transformation in $\bR^d$ with an invertible matrix $A\in \bR^{d\times d}$, and $b\in \bR^d$ is an offset vector. In that case, problem \eqref{eq:weights-EE-v1} also recovers the exact weights, and so does formulation \eqref{eq:weights-EE-v2-full} with $M(x)=A^T A$.

\section{Implementation}
\label{sec:implementation}

All the above routines require numerous computations of Wasserstein distances, barycenters, and the efficient computation of gradients with respect to distances and barycentric weights. To this end, we rely on a GPU implementation based on entropic regularization that enables fast automatic differentiation. We now recall the definition of the quantities that are computed by our solvers and provide some specific details about our implementation. Our code is available at the following address:
\begin{center}
\href{https://gitlab.tue.nl/20220022/sinkhorn-rom}{https://gitlab.tue.nl/20220022/sinkhorn-rom}
\end{center}

\subsection{Entropic regularization}

\paragraph{Discrete optimal transport:}
We work with discrete measures sampled on a 2D or 3D domain $\Omega$. We write two probability distributions $\ma, \mb \in \cP_2(\Omega)$ as weighted sums of Dirac masses:
\begin{equation*}
\ma = \sum_{i=1}^\textbf{N} \textbf{a}_i \delta_{\textbf{x}_i} \quad \text{and} \quad
\mb = \sum_{j=1}^\textbf{M} \textbf{b}_j \delta_{\textbf{y}_j} ~,
\end{equation*}
with sample locations $\textbf{x}_i$ and $\textbf{y}_j$ in $\mathbb{R}^d$
and where the vectors of weights $\textbf{a}_i$ and $\textbf{b}_j$ are non-negative and sum up to 1. We consider the squared Euclidean cost function $C(x,y)= \Vert x-y\Vert^2$ on the feature space $\mathcal{X}=\mathbb{R}^d$ and define the Monge-Kantorovich problem as:
\begin{equation}
\label{eq:discrete-OT}
\text{OT}(\ma,\mb) = \min_{\pi \in \Pi(\alpha,\beta)} \sum_{i,j} \pi_{i,j}C(\textbf{x}_i,\textbf{y}_j)
\end{equation}
where the set of admissible transport plans is defined by the constraints:
\begin{equation}
\quad \Pi(\ma,\mb) := \left\{ \pi \in \mathbb{R}_{+}^{N\times M} \quad \text{s.t.}: \pi\geq 0, ~\pi\mathbf{1}_M = \textbf{a},~ \pi^T\mathbf{1}_N = \textbf{b} \right\}~. 
\end{equation}
The Wasserstein-2 distance between the distributions $\ma$ and $\mb$ is defined as $W(\ma,\mb) := \sqrt{ \text{OT}(\ma,\mb)}$. We note that the solution of problem \eqref{eq:discrete-OT} is not necessarily unique, and that this large linear program may be hard to solve exactly for large values of \textbf{N}, \textbf{M} and $d$. 

\paragraph{Entropic optimal transport:}
In order to disambiguate this optimization problem while opening the door to fast parallel solvers, a common strategy is to add a smooth, strictly convex penalty to the linear objective $\langle \pi, C \rangle$ of~\eqref{eq:discrete-OT} \cite{peyre2019computational,feydy2020geometric}.
The entropy-regularized optimal transport problem reads:
\begin{align}\label{eq:OT_eps}
\text{OT}_{\varepsilon}(\ma,\mb) ~=~ \min_{\pi \in \Pi(\ma,\mb) } \langle\pi, C \rangle + 2 \,\varepsilon\, \text{KL} (\pi| \ma \otimes\mb )~,
\end{align}
where  $\text{KL}(u|v) =\sum_i u_i \log(u_i/v_i) -u_i +v_i$ denotes the Kullback-Leibler divergence  and  $\varepsilon>0$ is a hyperparameter that is homogeneous to the square of a distance and that we identify with the square of a blur radius $\sigma = \sqrt{\varepsilon}$. 
The convex optimization problem \eqref{eq:OT_eps} can be solved efficiently using fast iterative methods such as the Sinkhorn algorithm.

When $\varepsilon \rightarrow 0$, 
the entropic cost $ \text{OT}_{\varepsilon}(\ma,\mb)$ converges towards $\text{OT}(\ma,\mb)$.
However, for all positive values of $\varepsilon>0$, we must stress that $\text{OT}_{\varepsilon}$
does not induce a distance between probability distributions since $\text{OT}_{\varepsilon}(\mb,\mb) \neq 0$. To overcome this limitation and retrieve well-posed minimization problems for barycenters \cite{feydy2020geometric}, a common strategy is to use the de-biased Sinkhorn divergence defined by:
\begin{equation*}
S_{\varepsilon}(\ma,\mb)=  \text{OT}_{\varepsilon}(\ma,\mb)
-\tfrac{1}{2} \text{OT}_{\varepsilon}(\ma,\ma) -\tfrac{1}{2} \text{OT}_{\varepsilon}(\mb,\mb)~.
\end{equation*}
This formula defines a differentiable, positive and definite loss function that is convex with respect to each variable and behaves reliably for measure-fitting applications \cite{ramdas2017wasserstein,feydy2018global,feydy2019interpolating,feydy2020geometric,shen2021accurate}.

\subsection{Fast Sinkhorn barycenters}

Since the Sinkhorn divergence is convex with respect to both of its arguments, we can use it to approximate the Wasserstein barycenter problem as:
\begin{equation}
	\label{eq:sinkhorn_barycenter}
	\bary_\varepsilon(w_1, \dots, w_N\,;\, \barfun_1, \dots, \barfun_N)
	~=~
	\arg
	\min_{\beta \in \cP_2(\Omega)} \sum_{i=1}^N \weightcomp_i S_\varepsilon(\barfun_i, \beta)~.
\end{equation}
This problem is strictly convex with respect to $\beta$ when $\varepsilon > 0$ and admits a unique solution: entropic regularization smoothes out the technical difficulties that are caused by the genuine Wasserstein distance. Going further, this ``Sinkhorn barycenter'' problem has been studied in depth in \cite{janati2020debiased}, which proposed a de-biased Sinkorn solver.

In our numerical experiments, all probability distributions are sampled on a regular 2D grid with pixel size $\sigma$. We set a temperature $\varepsilon = \sigma^2$ for the entropic regularization and use Sinkhorn divergences and barycenters as drop-in replacements for the squared Wasserstein-2 distance and Wasserstein barycenters. As illustrated in \cite{feydy2020geometric}, such a small value of $\varepsilon$ corresponds to a negligible approximation error. Crucially, we rely on the fast solvers of the GeomLoss library \cite{feydy2019interpolating} to compute transport-related quantities, with gradients provided through automatic differentiation \cite{paszke2017automatic,feydy2020fast}.
These numerical routines implement a handful of recent advances which are critical to performance and accuracy:
\begin{enumerate}
\item As introduced in \cite{solomon2015convolutional}, we use a separable Gaussian convolution operator to reduce the time complexity of a Sinkhorn iteration from $O(N_\text{pixels}^2)$ to  $O(N_\text{pixels}^{3/2})$ in dimension $d=2$ and $O(N_\text{pixels}^{4/3})$ in dimension $d=3$.
\item As detailed in \cite{feydy2020geometric,feydy2020fast,charlier2021kernel}, we use the KeOps library to perform the Sinkhorn iterations in the log-domain with guaranteed numerical stability, optimal run times and a $O(N_\text{pixels})$ memory footprint on GPUs. 
\item As introduced in \cite{janati2020debiased}, we use de-biased iterations for the Sinkhorn barycenter problem.
\item As introduced in \cite{knight2014symmetry} and documented in \cite{feydy2020geometric}, we use symmetrized Sinkhorn iterations to accelerate convergence and guarantee that every call to a Sinkhorn divergence is invariant to the ordering of the input distributions.
\item Following ideas that were introduced in \cite{kosowsky1994invisible,merigot2011multiscale,levy2015numerical,schmitzer2019stabilized} and documented in \cite{feydy2020geometric}, we use an annealing and multiscale heuristic to speed up convergence of the Sinkhorn loops.
\end{enumerate}
These algorithmic ``tricks'' add up to a fast solver than we use to compute (approximate) Wasserstein barycenters of large collections of 2D and 3D distributions. We would like to note that even if the computation of Wasserstein barycenters is a provably hard problem in high-dimensional spaces \cite{altschuler2022wasserstein}, it is relatively easy to solve up to a set tolerance in spaces of dimension 2 and 3 \cite{altschuler2021wasserstein}. Working with distributions that are sampled on a fixed image grid alleviates the complex issue of the identification of the support of the Wasserstein barycenter and lets us focus on the (easy, convex) problem of finding optimal point masses for $\beta$ that minimize \eqref{eq:sinkhorn_barycenter}. 


\section{Numerical experiments}
\label{sec:numerical_experiments}
We now illustrate the performance of the different projection and regression methods that were discussed in this paper. The discussion is carried either on pedagogical toy datasets, or on relatively simple examples from Model Order Reduction. We divide our experiments in two main categories:
\begin{enumerate}
\item Best $n$-term approximation (from Section \ref{sec:algo-best-n-term}): we study the convergence and final accuracy of the proposed descent algorithms for projection on $\cA_N^n$.
\item Regression with sparse barycentric interpolation (from Section \ref{sec:interp-weights}):
\begin{itemize}
	\item Behavior of the baseline sparse methods that we presented in Section~\ref{sec:other-interp}.
	\item Performance of the adaptive, sparse strategy (\AS) that we proposed in Section~\ref{sec:LEE}.
	\item The best $n$-term approximation cannot be computed in real application scenarios, but serves as a benchmark for our regression methods as it gives the optimal performance that can be achieved with a barycenter of $n$ distributions from the training dataset.
\end{itemize}
\end{enumerate}

As a guiding example for our tests, we consider elements from $\Pr$ generated from a parametric two-dimensional viscous Burgers' equation. For all times $t\in [0, T]$ and all points $r\in \Omega=\bR^2$, we consider a solution $u(t, r)$ such that:

\begin{align*}
\partial_t u + \frac{1}{2}\partial_{r_1}(u^2) + \frac{1}{2}\partial_{r_2}(u^2) ~=~ \mb\Delta_r u~.
\end{align*}
To specify the initial condition, we define a square centered at a given point $c^\circ = (c^\circ_1, c^\circ_2)\in \bR^2$ with side length $w>0$:
$$
\rR = \{ r=(r_1,r_2)\in \bR^2\cond r_i \in [c^\circ_i - w/2, c^\circ_i+w/2],\, i=1,2 \}~.
$$
Then, we consider as initial condition to Burgers' equation the probability distribution:
\begin{align*}
u(0,r) = w^{-2}\charfun_{\rR} (r) 
= \begin{cases}
	| w|^{-2}\quad &\text{ if } r\in R, \\
	0\quad &\text{ otherwise.}
\end{cases}
\end{align*}
Solutions to this problem are nonnegative and belong to the space of integrable density functions $L^1(\bR^2)$. Mass is preserved in the sense that $\int_{\bR^2} u(t, r)\dr=1$ for all $t\in [0, T]$. Therefore for every $t\in[0, T]$, we understand the solution $u(t,\cdot)$ as the probability density of the measure $\rd \ma(t) = u(t, \cdot)\dr$ with $\dr$ being the Lebesgue measure. With a slight abuse of notation, we will say that $u(t,\cdot) \in \Pr$.

In the experiments below, we handle $t, c^\circ, w$ and $\mb$ as the parameters of our PDE solver. Vectors of input parameters read:
$$
x = (t, c^\circ_1, c^\circ_2, w, \mb)~.
$$
As parameter domain, we use:
\begin{equation*}
	X = \left\{x\in \bR^5: x \in [0, 5]\times [2,6]\times [2,6] \times [1,2]\times [5. 10^{-5},10^{-1}] \right\}.
\end{equation*}
For each $x\in X$, we then consider the associated solution (or ``snapshot''):
$$
y(x) \coloneqq u(t, \cdot\,; c^\circ, w, \mb) \in \cY=\Pr
$$
which is a probability measure in $\Pr$. Using a standard finite volume discretization, we can solve the PDE and generate a set of solutions:
$$
\rY \coloneqq \{ y(x) \in \Pr\cond x\in X \} \subset \cY~.
$$
In Figure~\ref{fig:some_fields}, we display some snapshots from $\rY$ that will be used as part of our training dataset $\barsetsp_N$. It is interesting to note that some snapshots look similar to each other, which illustrates possible redundancies in real-life datasets.

\begin{figure}[t]
	\centering
	\includegraphics[scale=0.6]{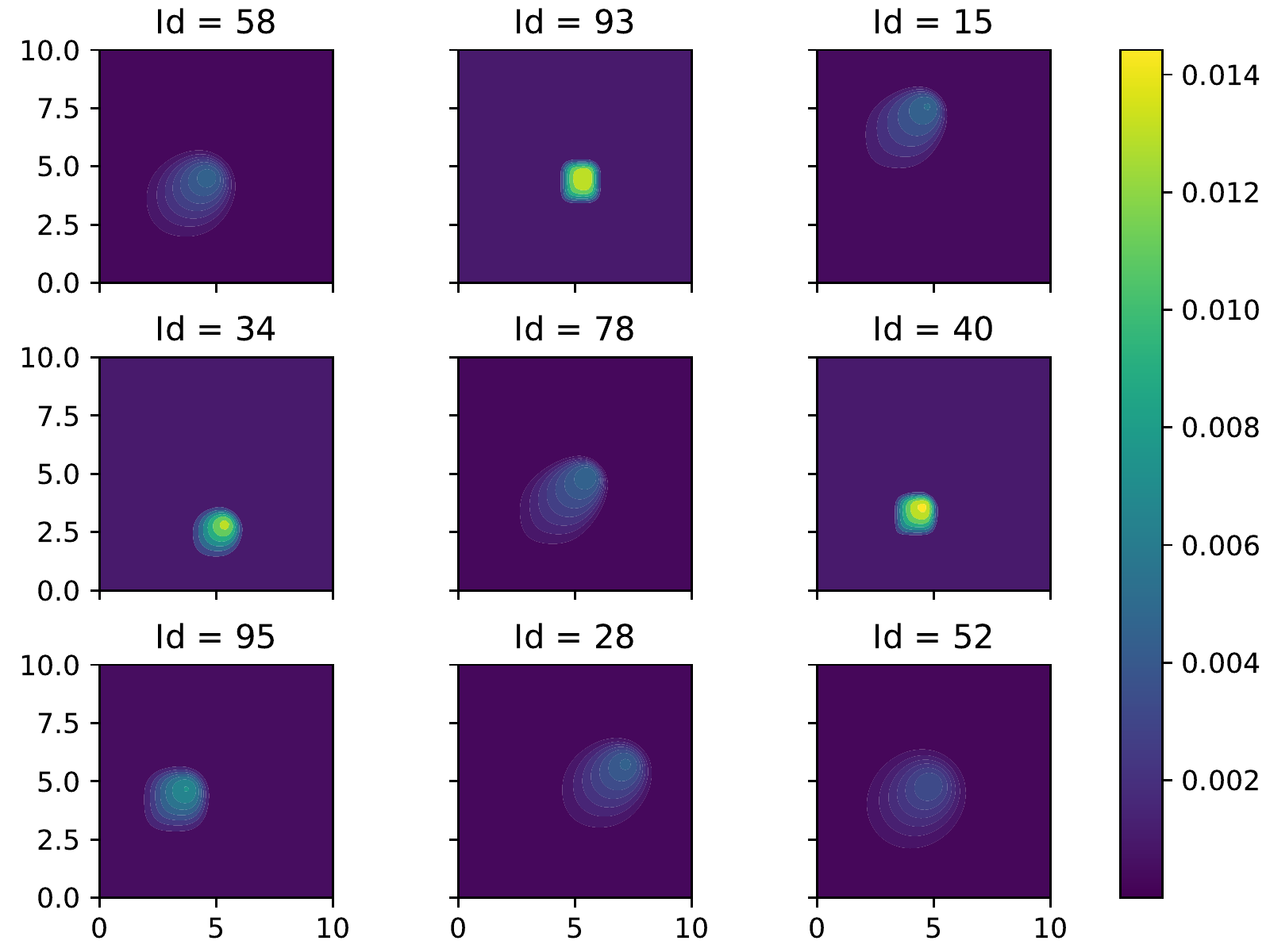}
	
	\caption{Some snapshots from the training set $\barsetsp_N$ generated by a numerical solver for Burgers' equation.}
	\label{fig:some_fields}
\end{figure}

\begin{figure}[p]
	\centering
	\subfloat[Reference distribution and predicted barycenter.]{\includegraphics[scale=0.46]{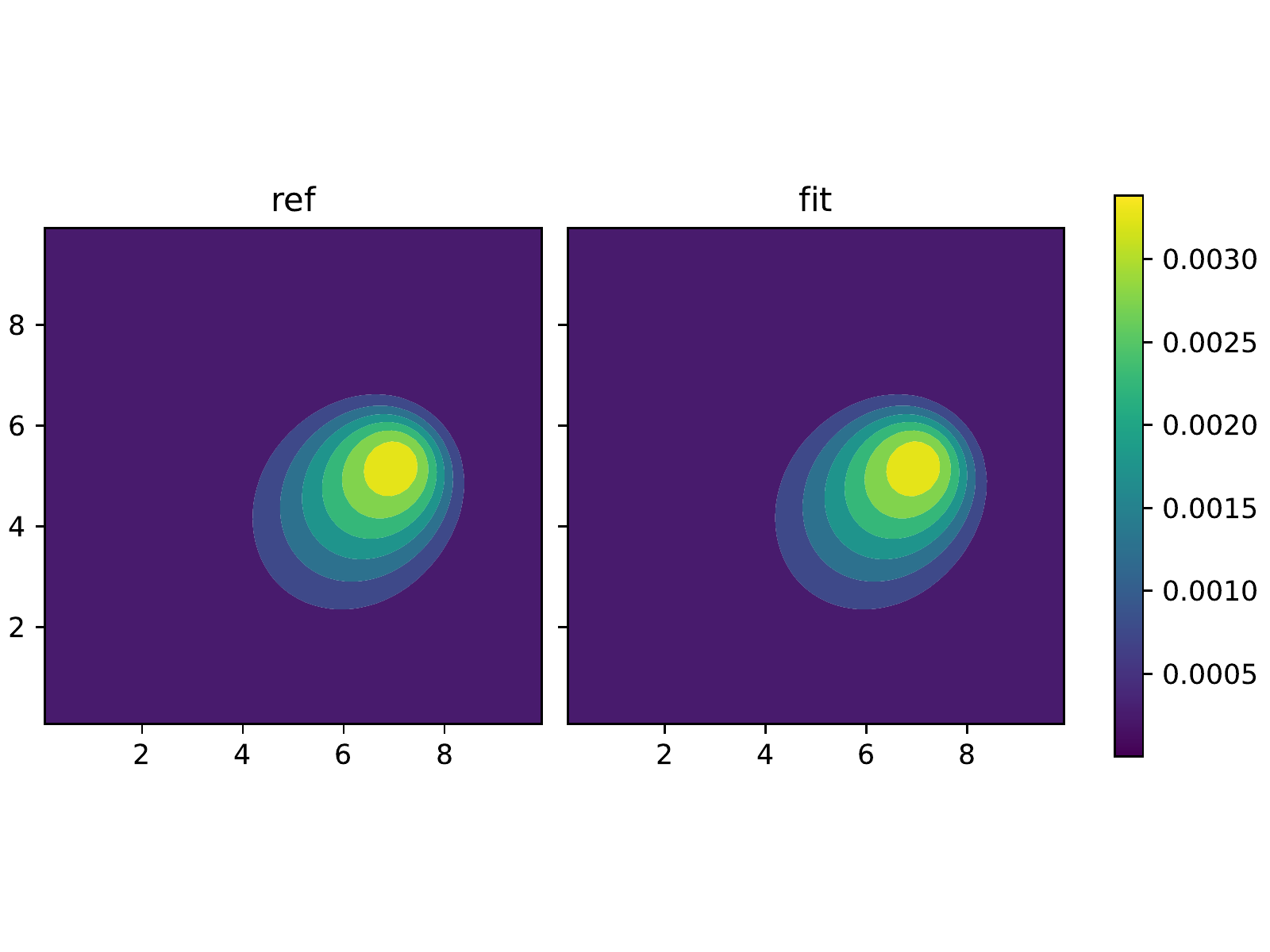}}\hspace{1cm}
	\subfloat[Predicted vector of weights.]{\includegraphics[scale=0.45]{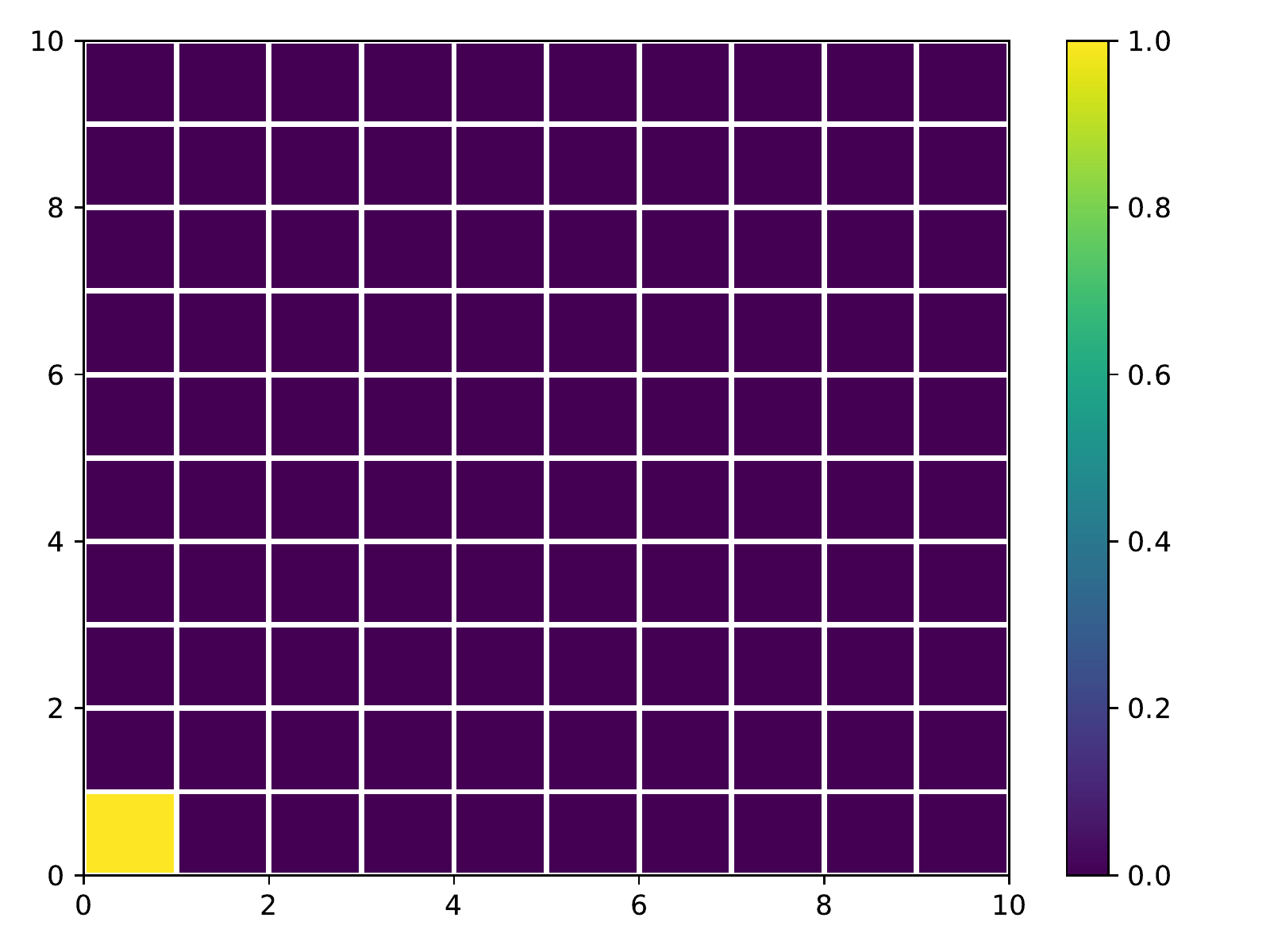}}
	\caption{ \textbf{Test 1.1:} We recover the exact vector of weights $(1,0,0,\dots)$.}
	\label{fig:PGD_in_prediction_weight}
\end{figure}

\begin{figure}[p]
	\centering
	\subfloat[Objective function.]{\includegraphics[scale=0.4]{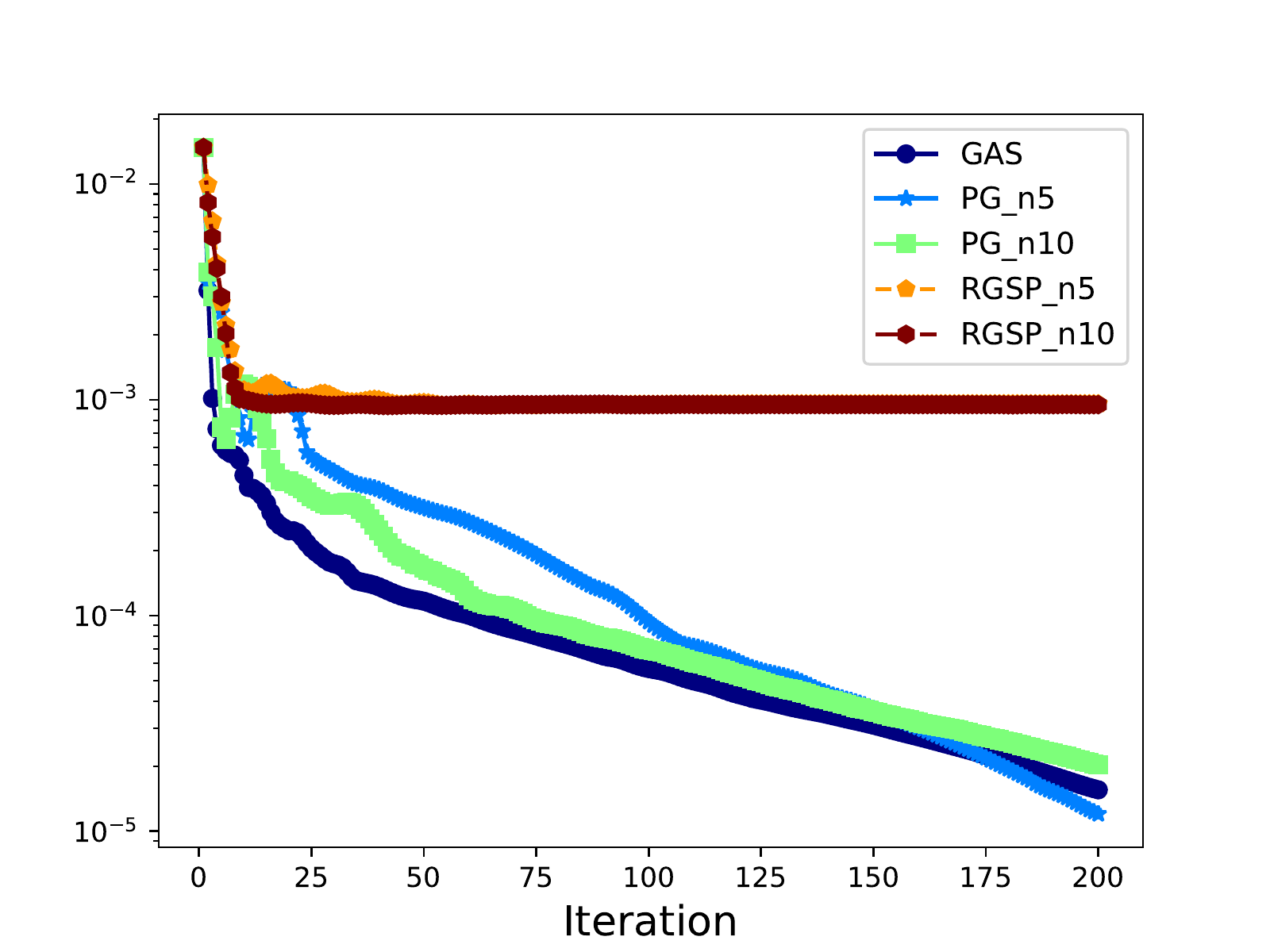} \label{fig:PGD_in_loss_support_a}}\hspace{1cm}
	\subfloat[Cardinal of the support.]{\includegraphics[scale=0.4]{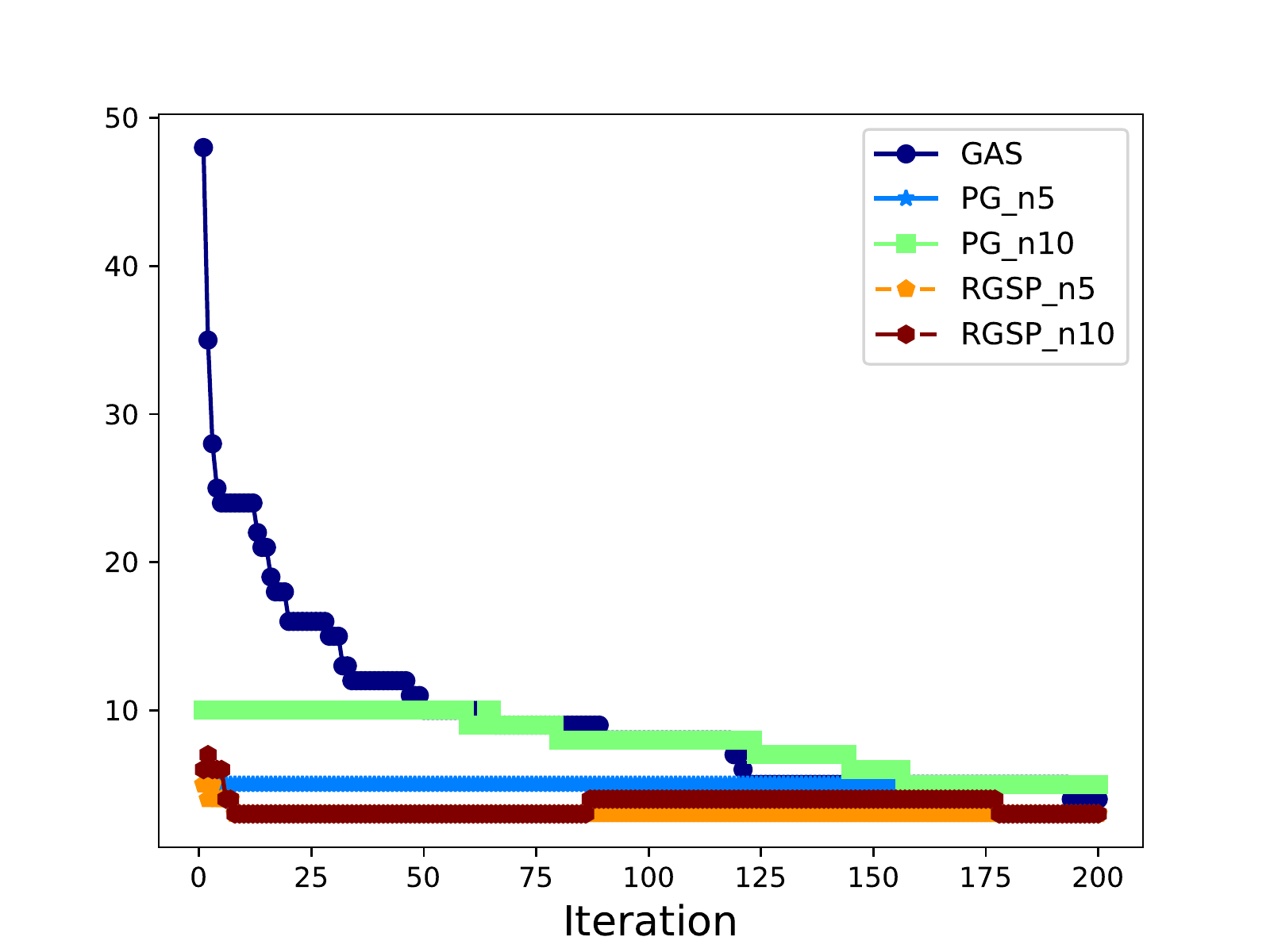} \label{fig:PGD_in_loss_support_b}}
	\caption{ \textbf{Test 1.1:} Converge of the iterative projection algorithms when the target distribution is one of the training snapshots.}
	\label{fig:PGD_in_loss_support}
\end{figure}

\begin{figure}[p]
	\centering
	\subfloat[Objective function.]{\includegraphics[scale=0.45]{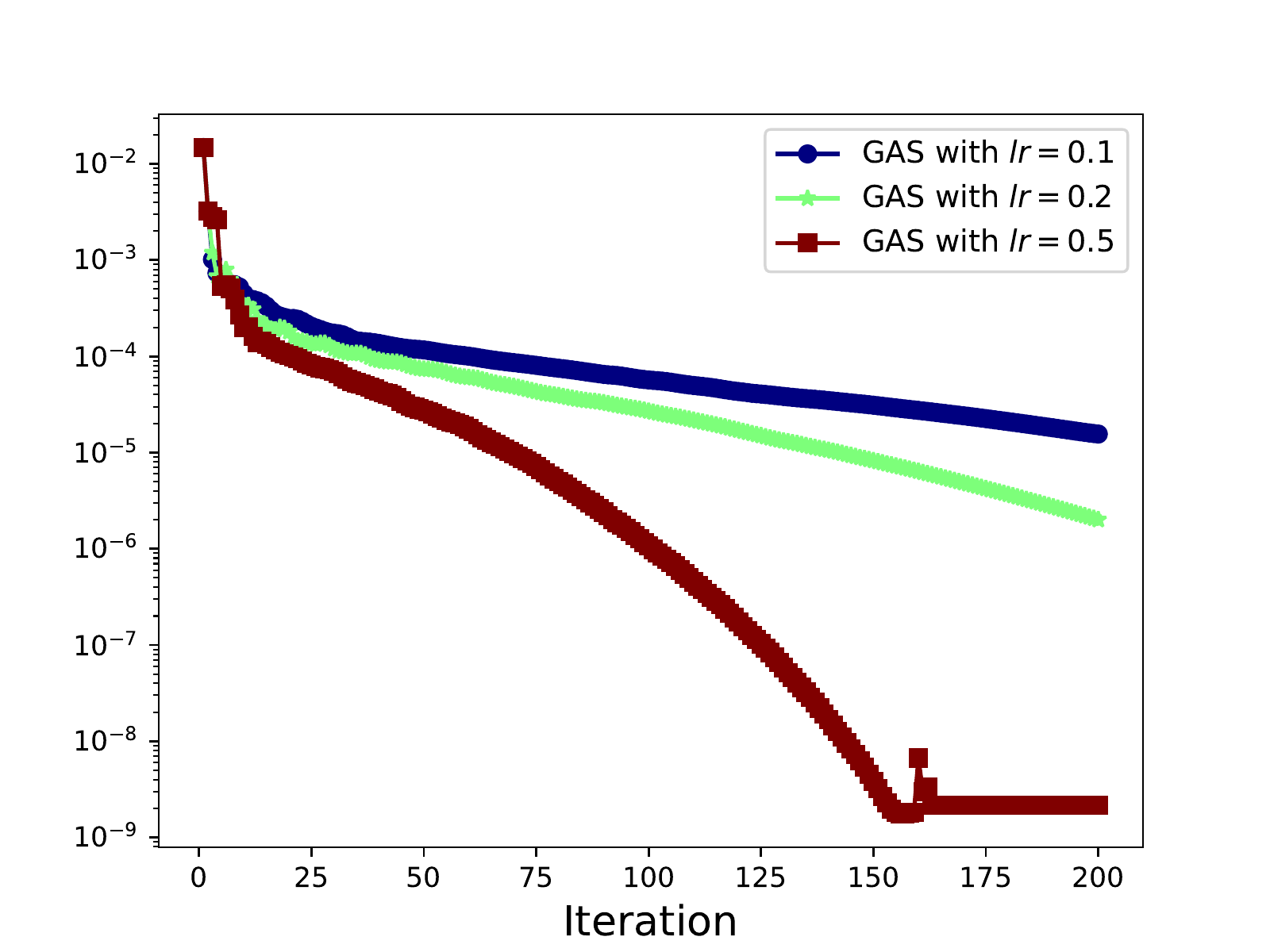}}\hspace{1cm}
	\subfloat[Cardinal of the support.]{\includegraphics[scale=0.45]{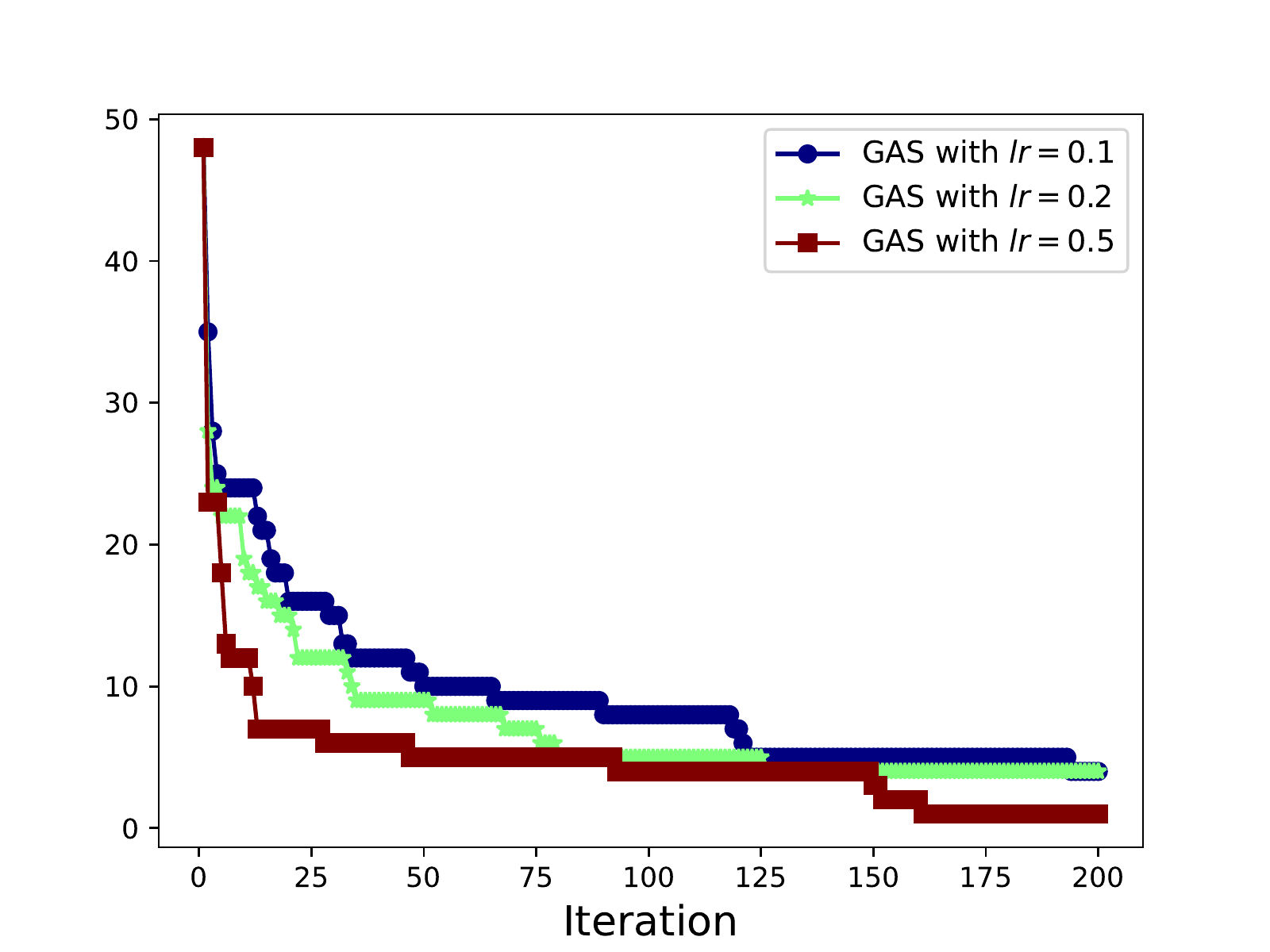}}
	\caption{ \textbf{Test 1.1:} Impact of the learning rate on convergence speed.}
	\label{fig:PGD_in_lr}
\end{figure}

\subsection{Best $n$-term barycentric approximation}

In this section, we study the performance of the algorithms presented in Section~\ref{sec:algo-best-n-term} to compute a best $n$-term barycentric approximation. Recall from \eqref{eq:best-n-term-barycenter} that this task consists in approximating some target measure $\alpha\in \Pr$ with the $n$-sparse barycenter:
\begin{equation}
P_{\cA_N^n}(\ma) \in \argmin_{\mb \in \class_N^n} W_2^2(\ma, \mb)~,
\end{equation}
where $\barsetsp_N=\{y_1, \dots, y_N\}$ is a dataset of $N$ snapshot measures from $\rY$, and $n\leq N$. As stated in \eqref{eq:best-n-term-lambda}, computing $P_{\cA_N^n}(\ma)$ boils down to finding the best $n$-sparse barycentric weights
\begin{equation}
\weight^n_N(\ma) \in \argmin_{\weight^n_N \in \Sigma_N^n} W_2^2(\ma, \bary(\weight_N^n, \barsetsp_N))~.
\end{equation}
We test the ability of algorithms \PG{}, \GAS{} and \RGSP~to solve problem \eqref{eq:best-n-term-lambda-2} in scenarios of increasing difficulty.

\begin{figure}[p]
	\centering
	\subfloat[Objective function.]{\includegraphics[scale=0.45]{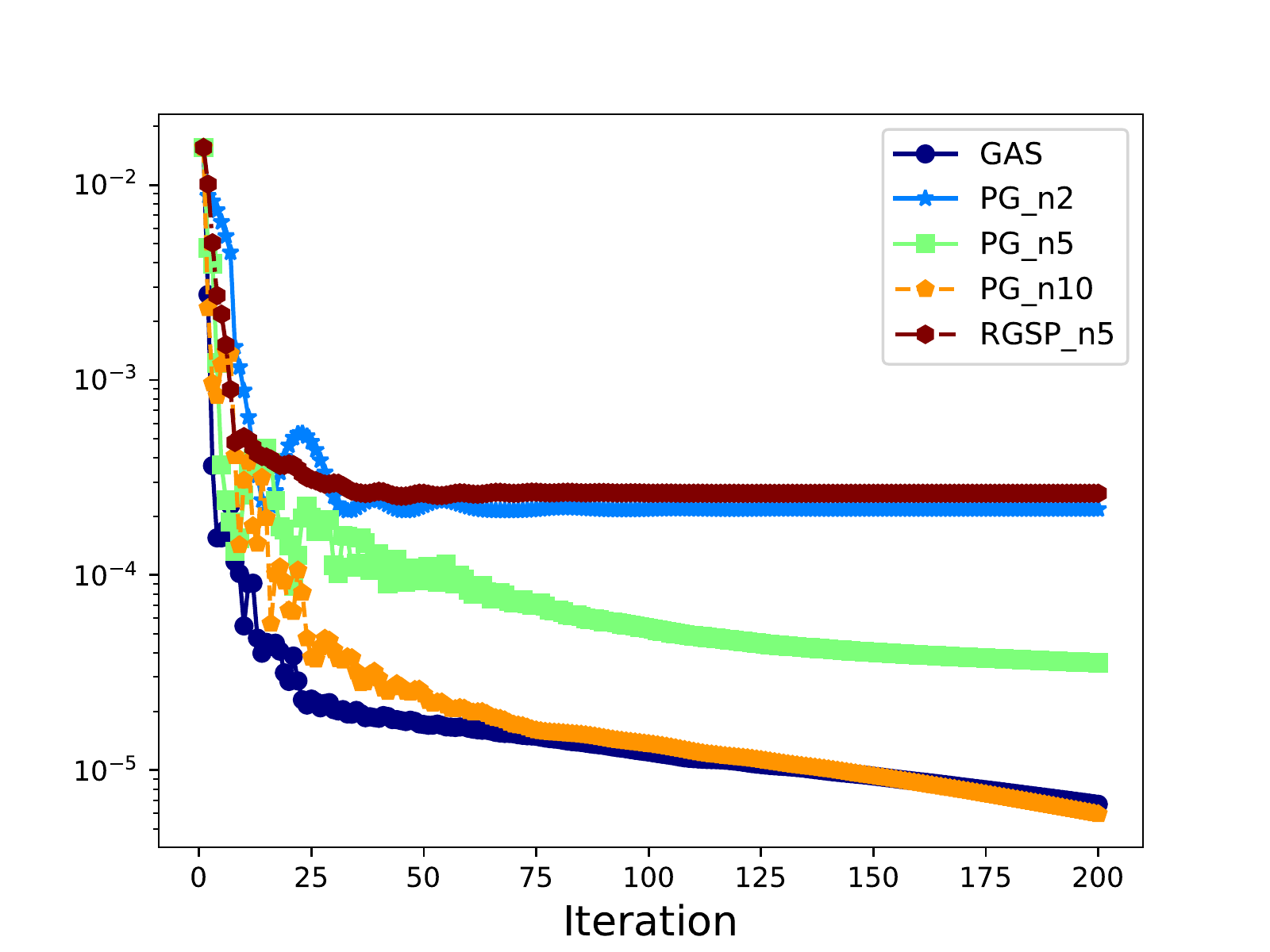}}\hspace{1cm}
	\subfloat[Cardinal of the support.]{\includegraphics[scale=0.45]{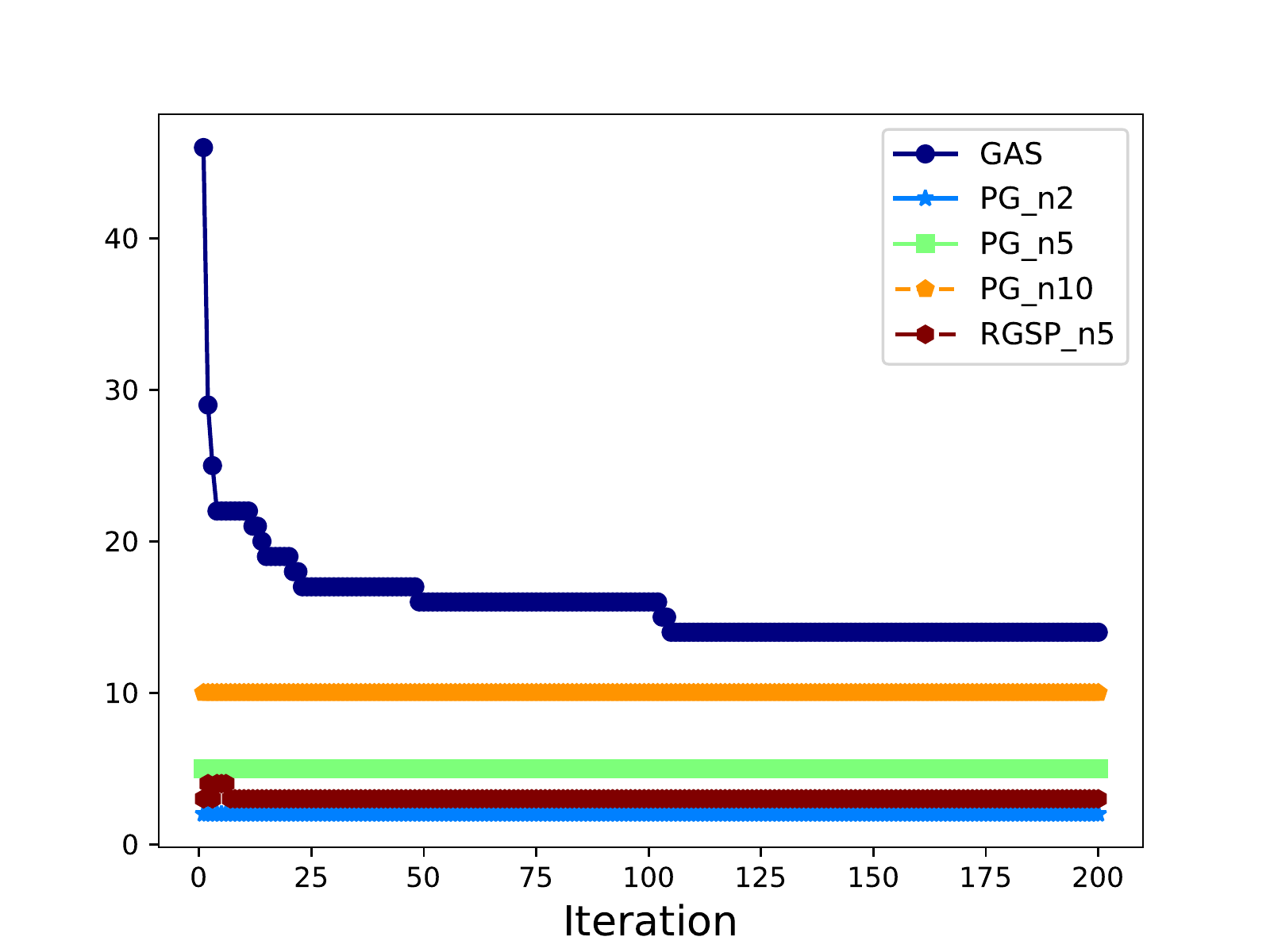}}
	\caption{\textbf{Test 1.2:} Convergence of the iterative projection algorithms when the target function is the Wasserstein barycenter of $2$ snapshots from the training dataset $\barsetsp_N$.}
	\label{fig:PGD_bary2_loss_support}
	
\end{figure}

\begin{figure}[p]
	\centering
	\subfloat[Reference distribution and predicted barycenter.]{\includegraphics[scale=0.45]{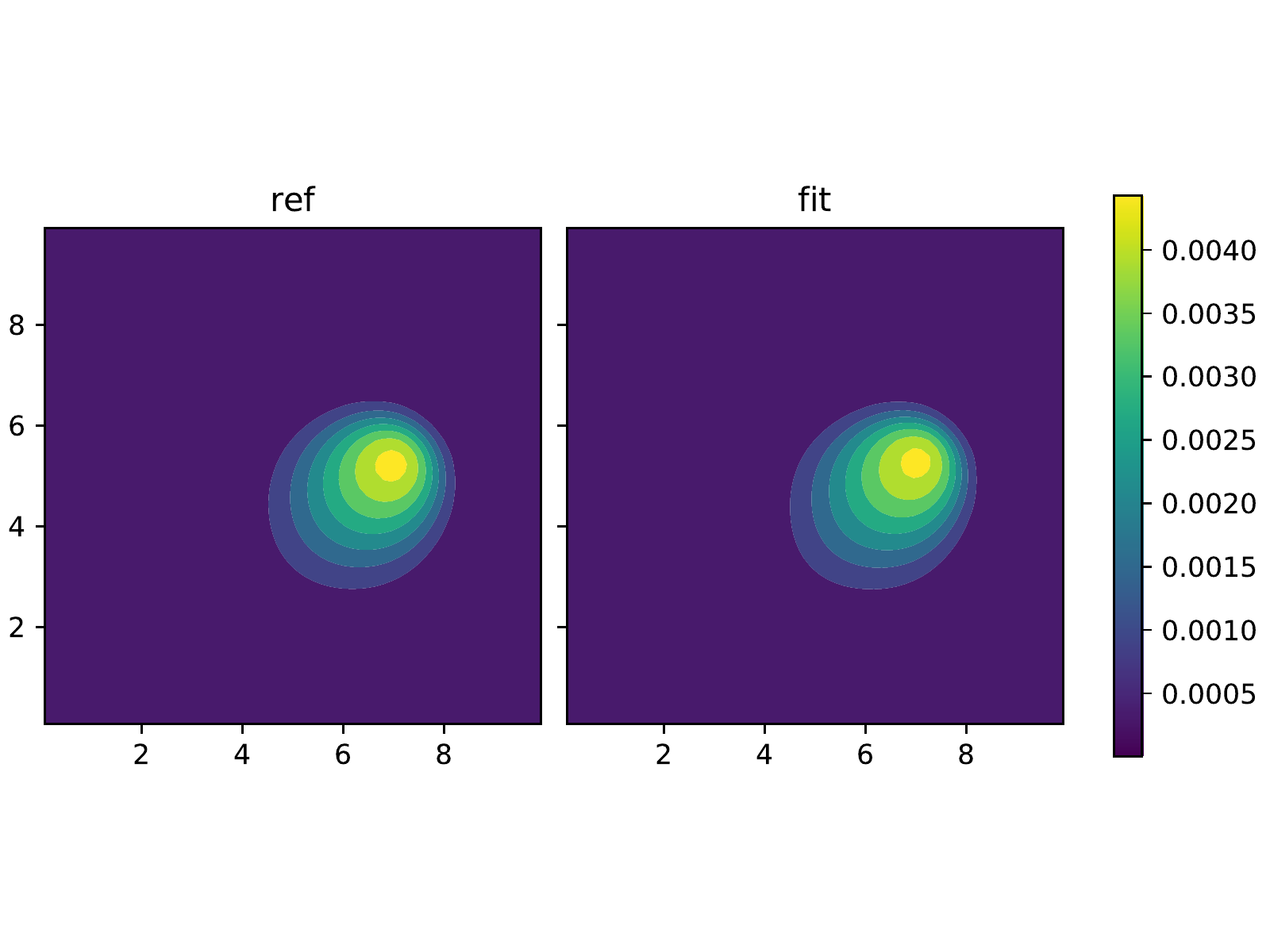}}\hspace{1cm}
	\subfloat[Predicted vector of weights.]{\includegraphics[scale=0.45]{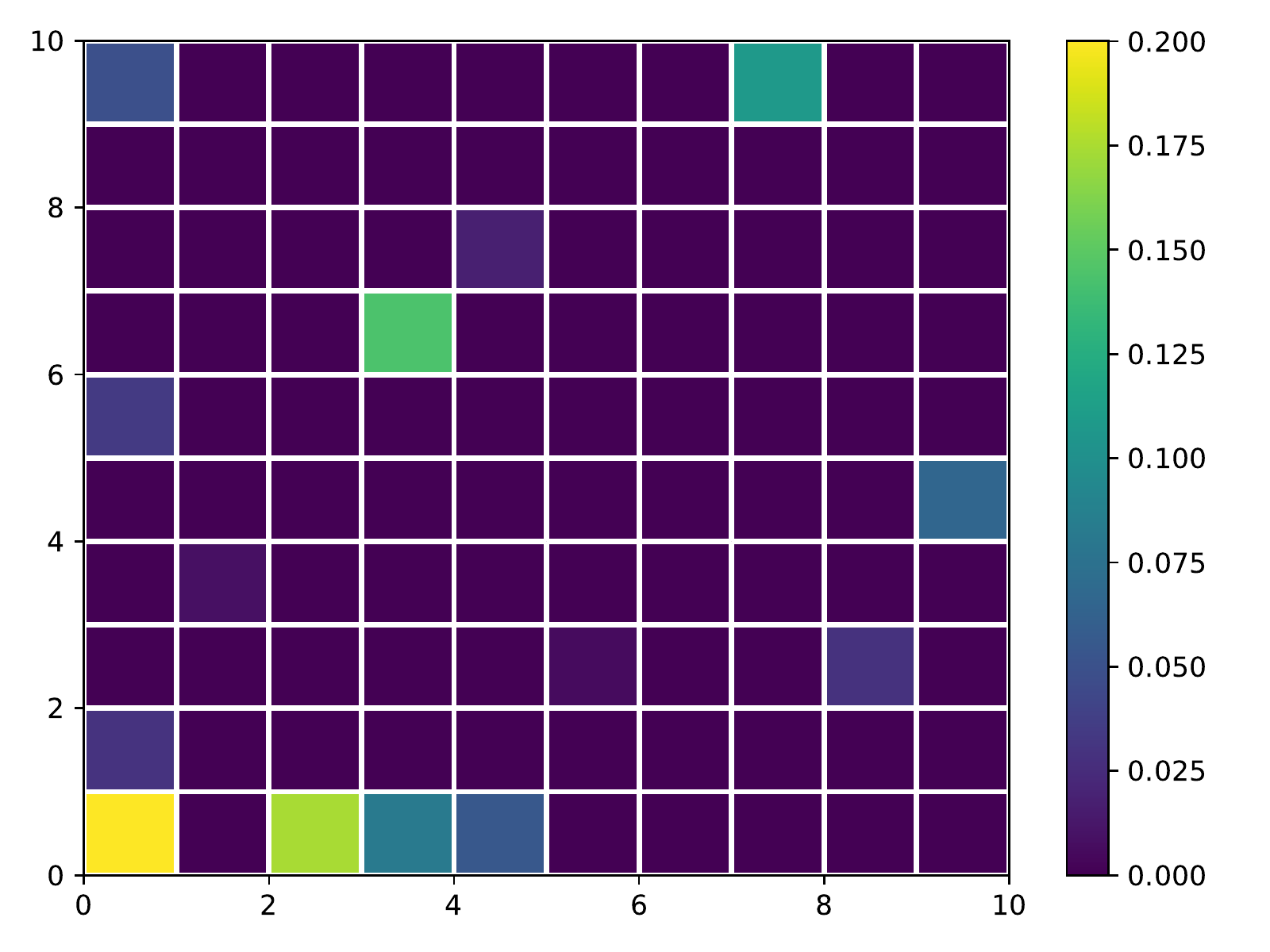}}
	\caption{\textbf{Test 1.2:} We recover a near-perfect fit to the reference distribution, but redundancy in the training dataset prevents us from recovering the expected vector of weights $(0.3, 0, 0.7, 0, 0, \dots)$.}
	\label{fig:PGD_bary2_prediction_weight}
\end{figure}

\begin{figure}[p]
	\centering
	\subfloat[Objective function.]{\includegraphics[scale=0.45]{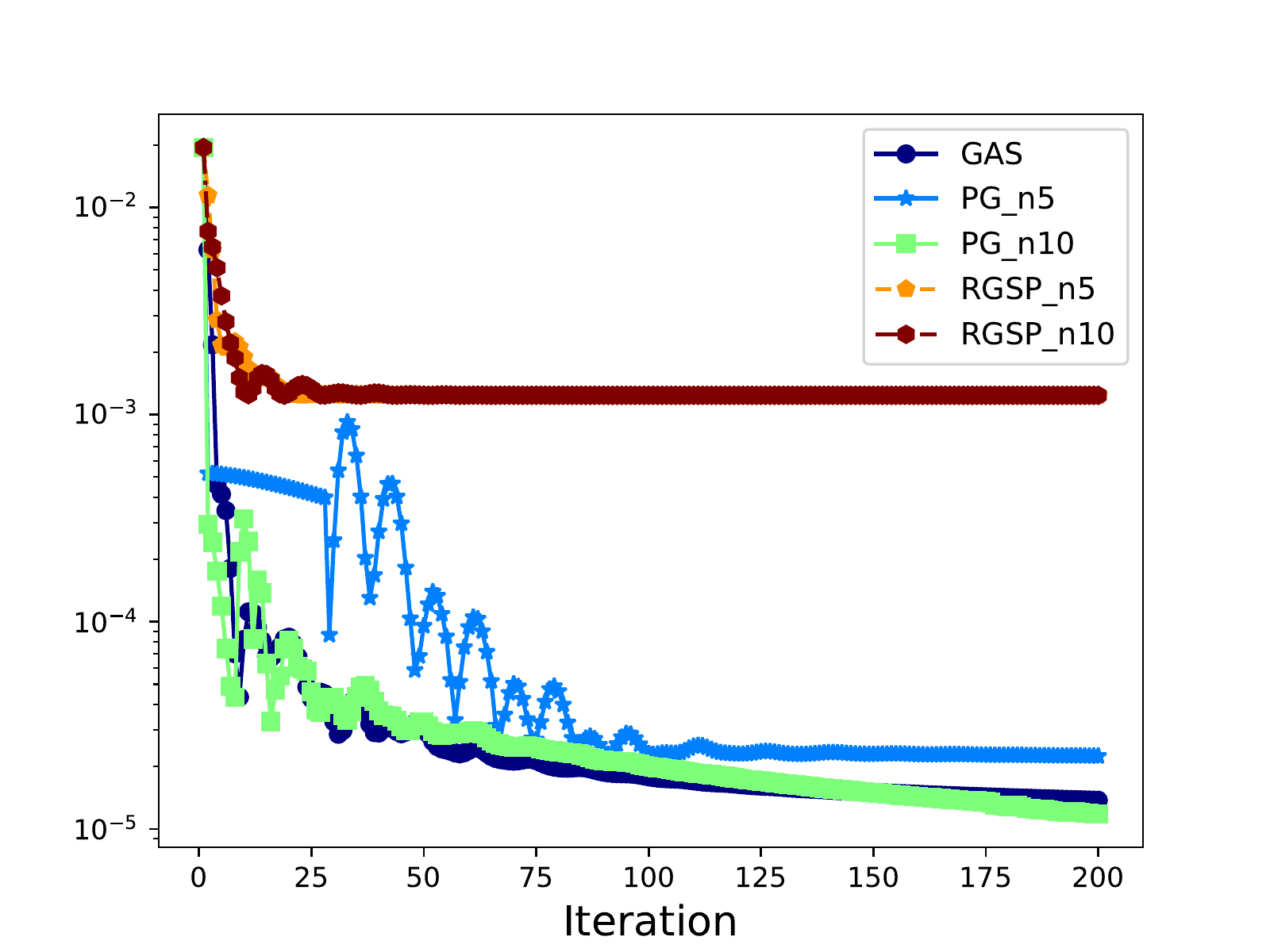}}\hspace{1cm}
	\subfloat[Cardinal of the support.]{\includegraphics[scale=0.45]{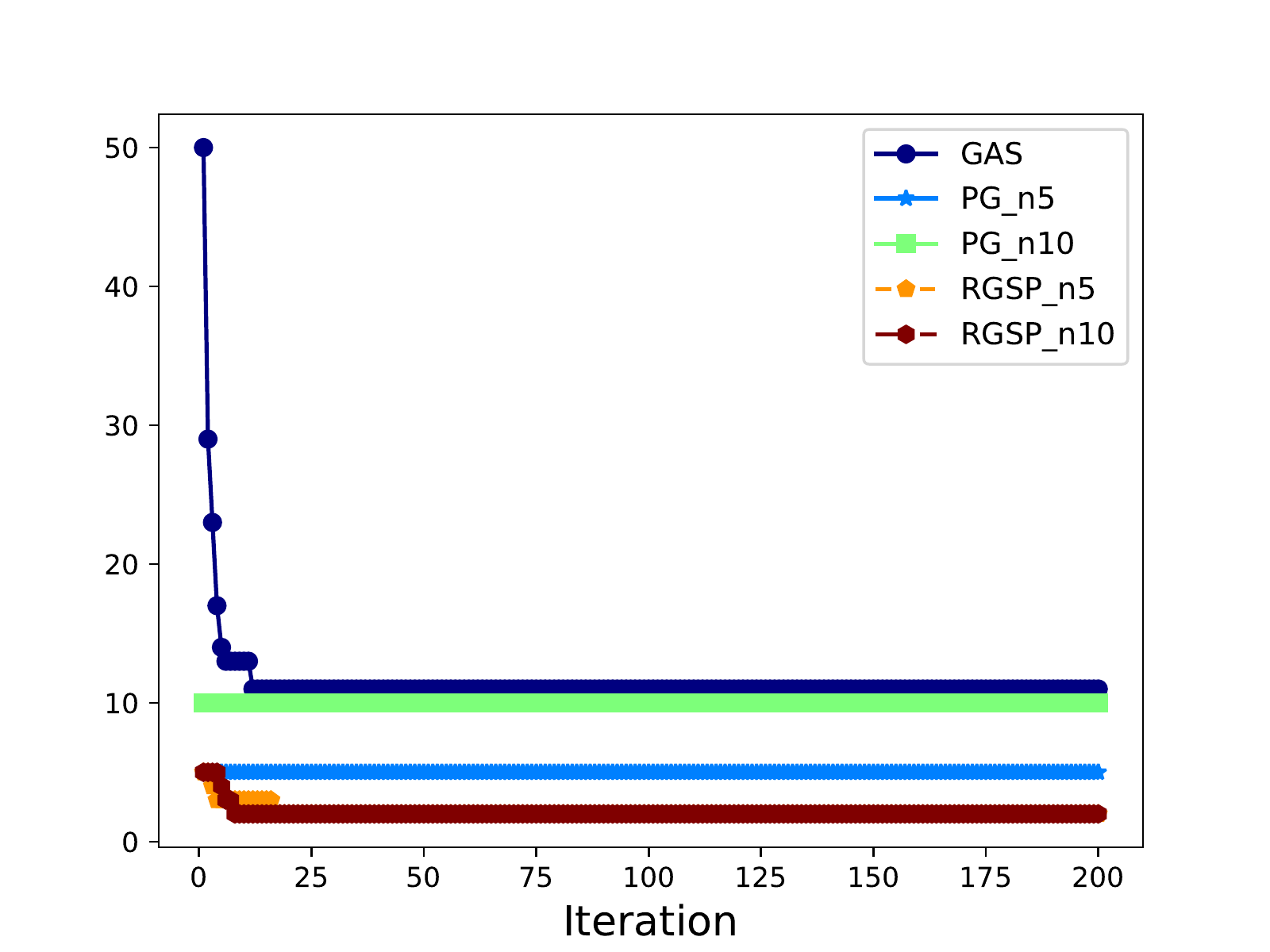}}
	\caption{ \textbf{Test 1.3:} Convergence of the iterative projection algorithms when the target distribution does not belong to the convex hull of the training dataset.}
	\label{fig:PGD_out_loss_support}
\end{figure}

\begin{figure}[p]
	\centering
	\subfloat[Reference distribution and predicted barycenter.]{\includegraphics[scale=0.45]{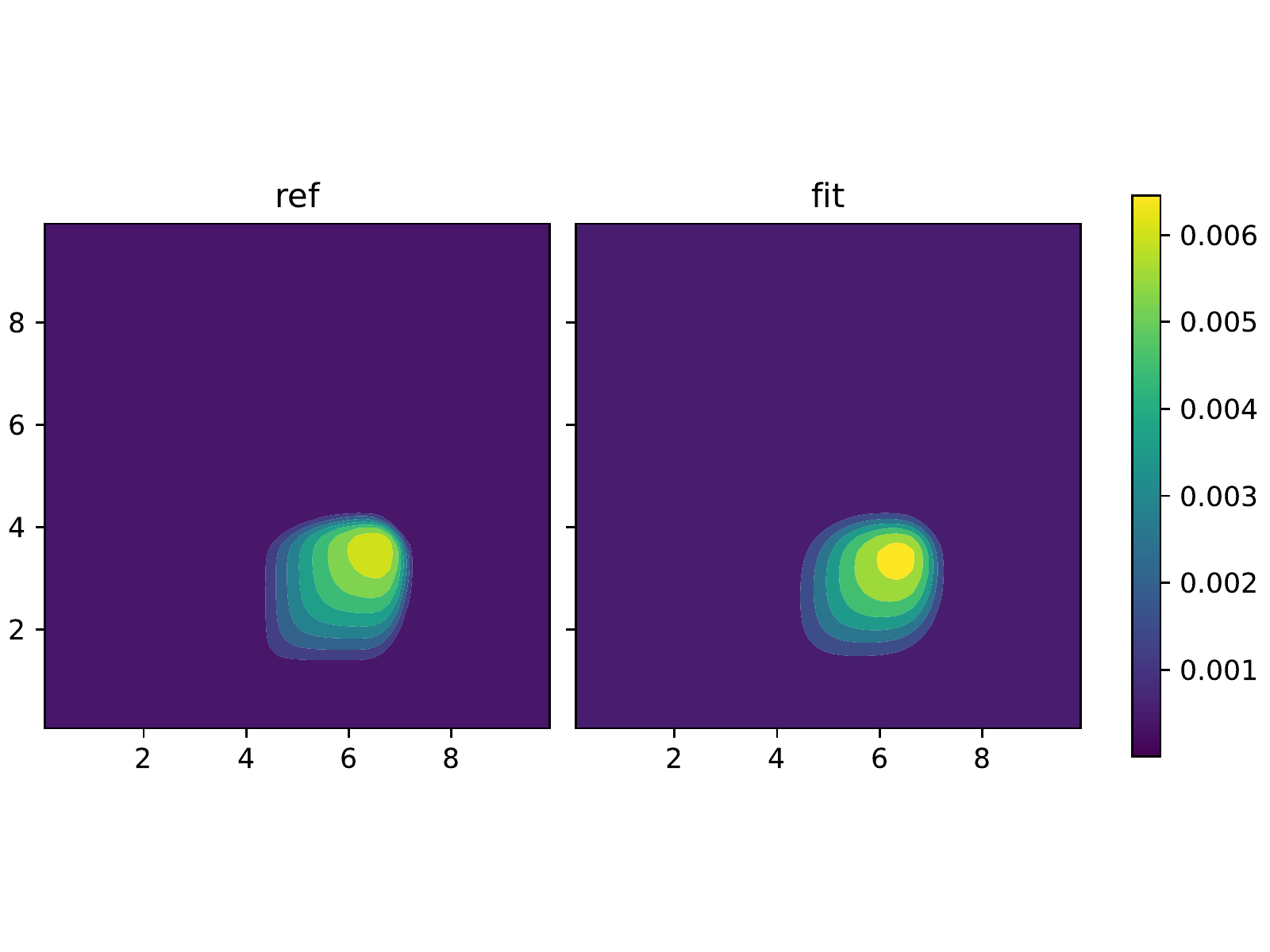}}\hspace{1cm}
	\subfloat[Predicted weights.]{\includegraphics[scale=0.45]{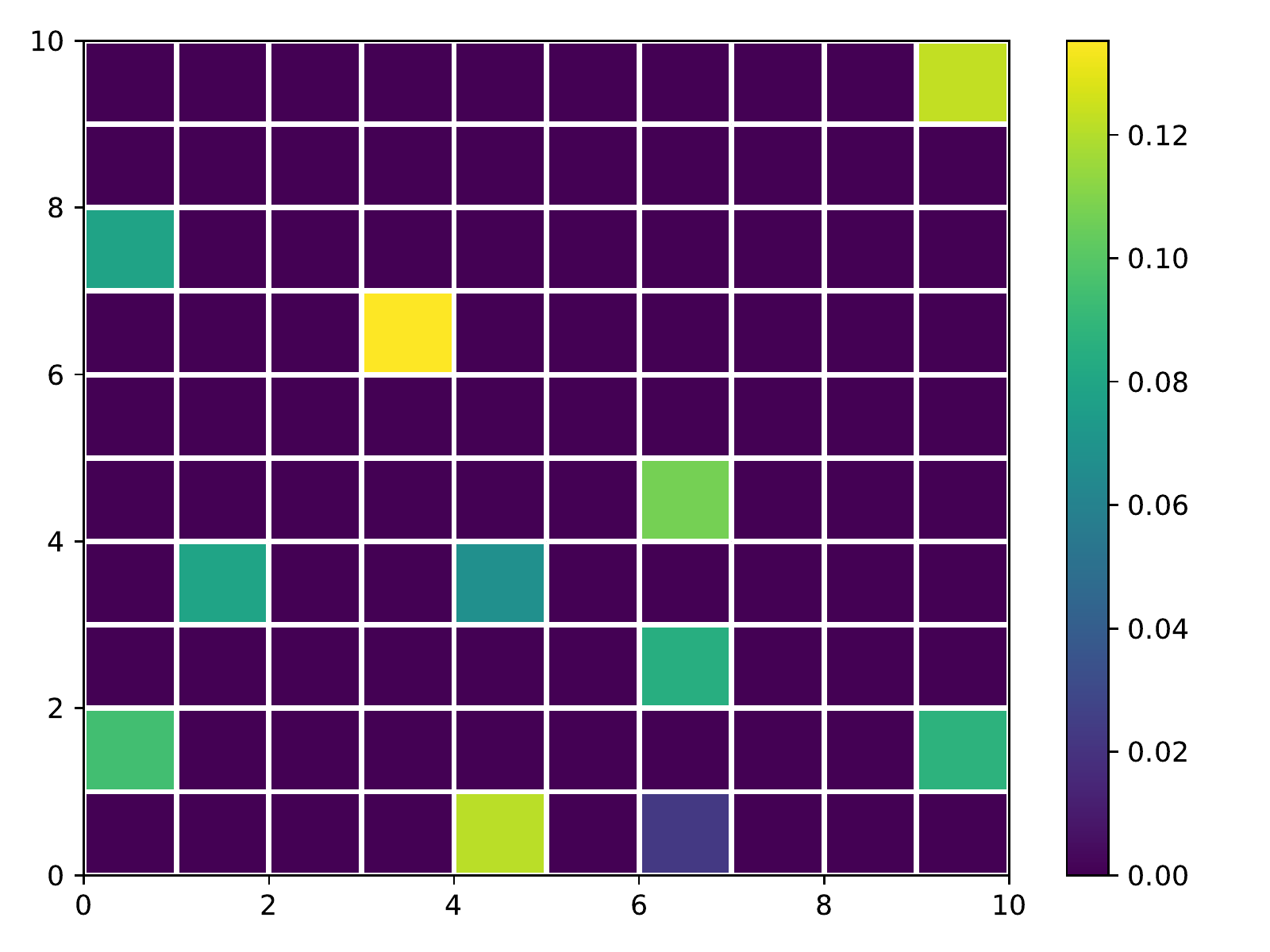}}
	\caption{\textbf{Test 1.3:} Approximation of a reference distribution that does not belong to the convex hull of the training dataset.}
	\label{fig:PGD_out_prediction_weight}
\end{figure}

\begin{figure}[p]
	\centering
	\subfloat[Naive guess: $M(x)=\Id$.]{\includegraphics[scale=0.4]{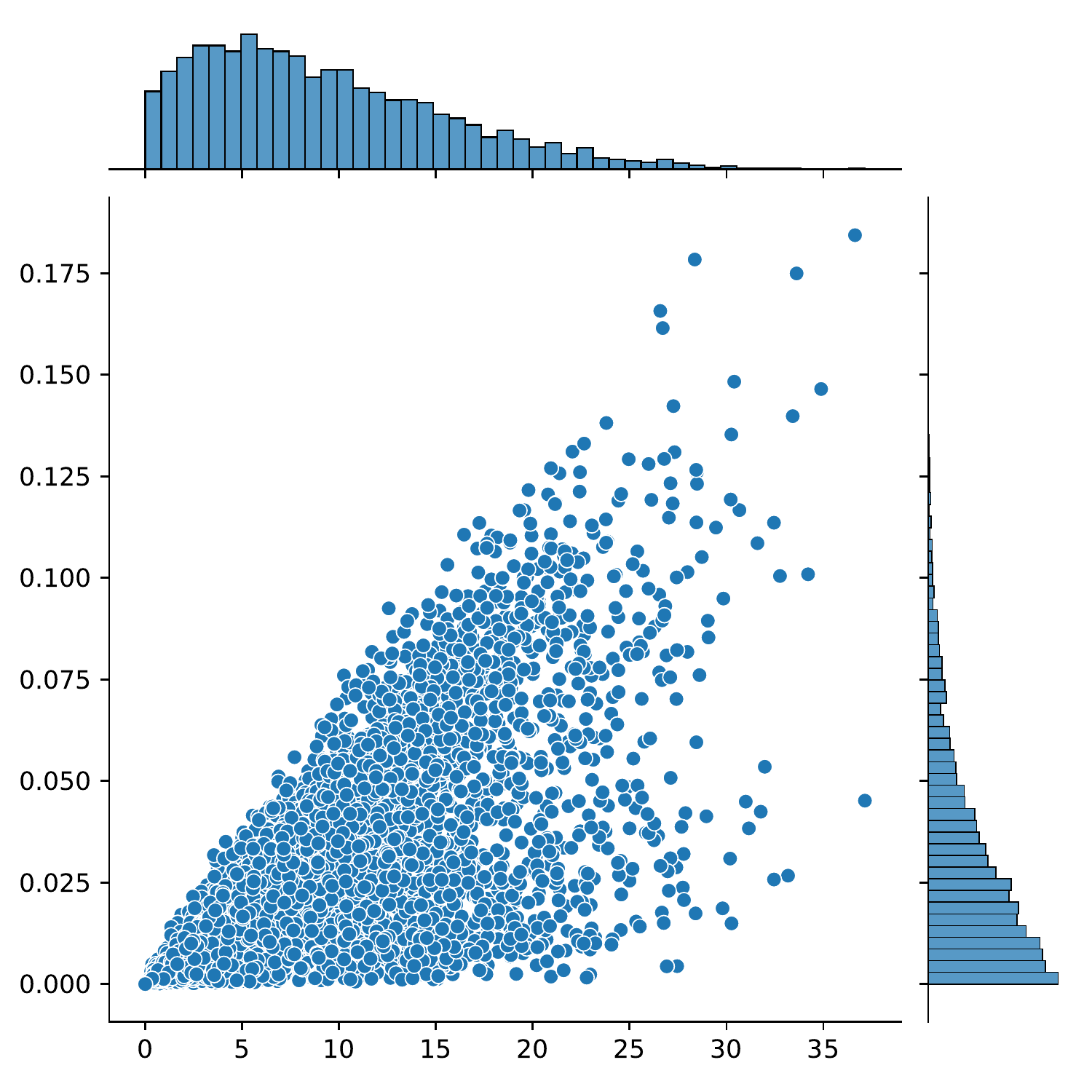}}\hspace{0.2cm}
	\subfloat[$M(x)$ after learning the local Euclidean metric with \eqref{eq:learned-EE}.]{\includegraphics[scale=0.4]{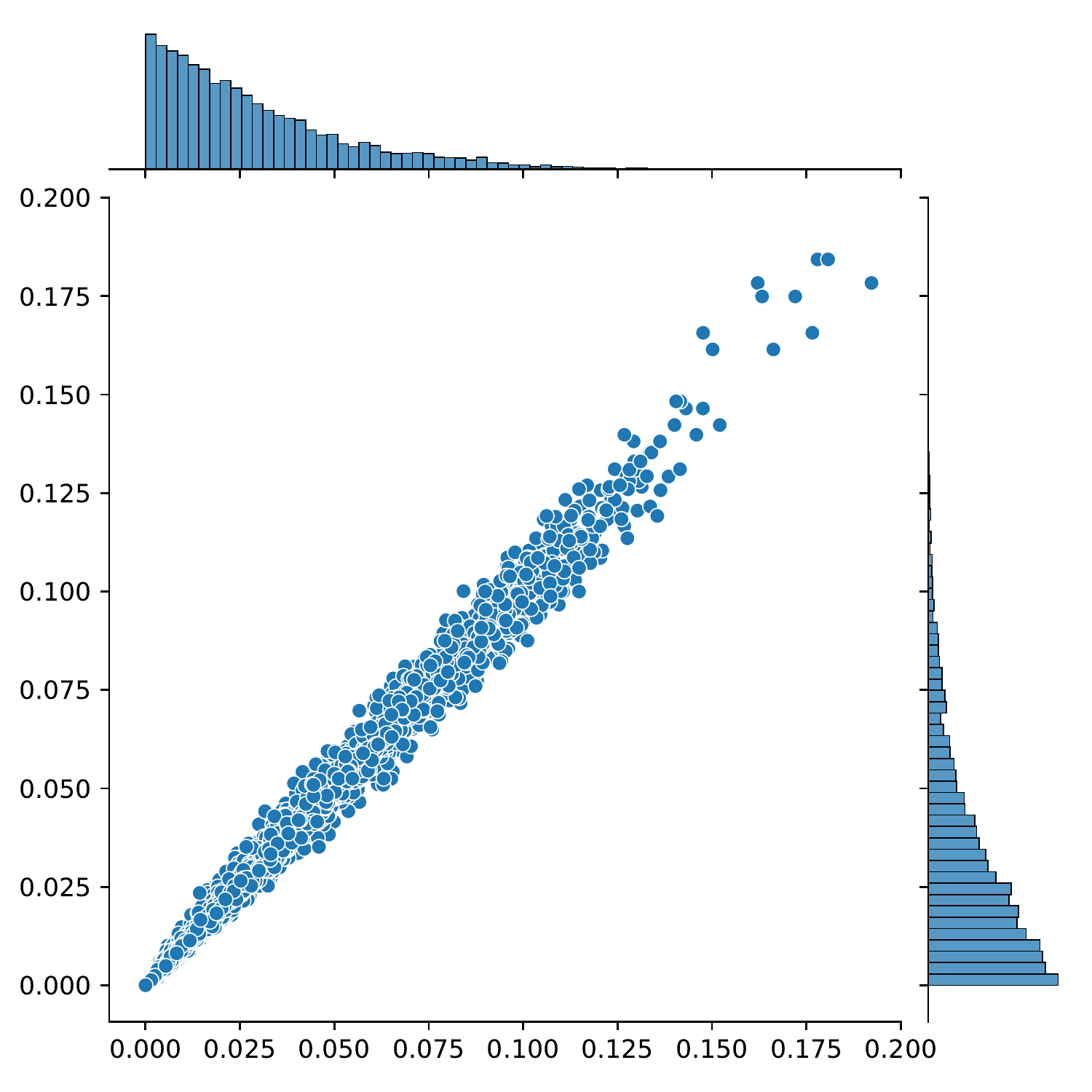}}
	\caption{Local Euclidean metric: joint plot between $(x- \tilde x)^T M(x) (x- \tilde x)$ and $W_2( f(x), f(\tilde x))^2$ for all $(x, \tilde x) \in \rX_N\times \rX_N$. (b) After learning $M(x)$, points cluster around the line $y=x$, revealing a good fit of the local Euclidean metric $(x- \tilde x)^T M(x) (x- \tilde x)$ to the squared Wasserstein distance $W_2( f(x), f(\tilde x))^2$.}
	\label{fig:LE}
\end{figure}

\begin{figure}[p]
	\centering
	\includegraphics[scale=0.4]{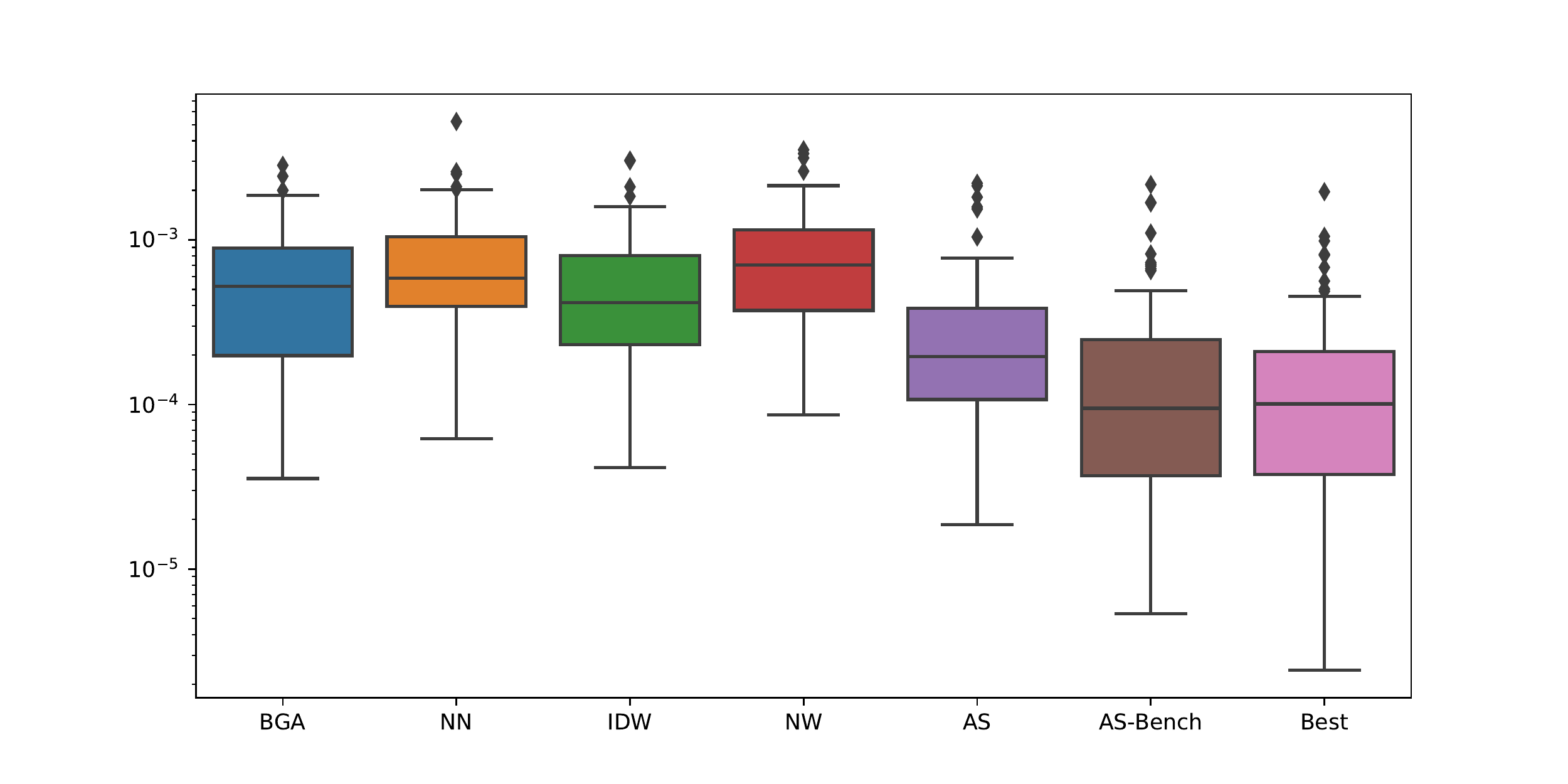}
	\caption{Boxplot for approximation errors in the validation set.}
	\label{fig:boxplot_MOR}
\end{figure}
\begin{figure}[p]
	\centering
	\includegraphics[scale=0.6]{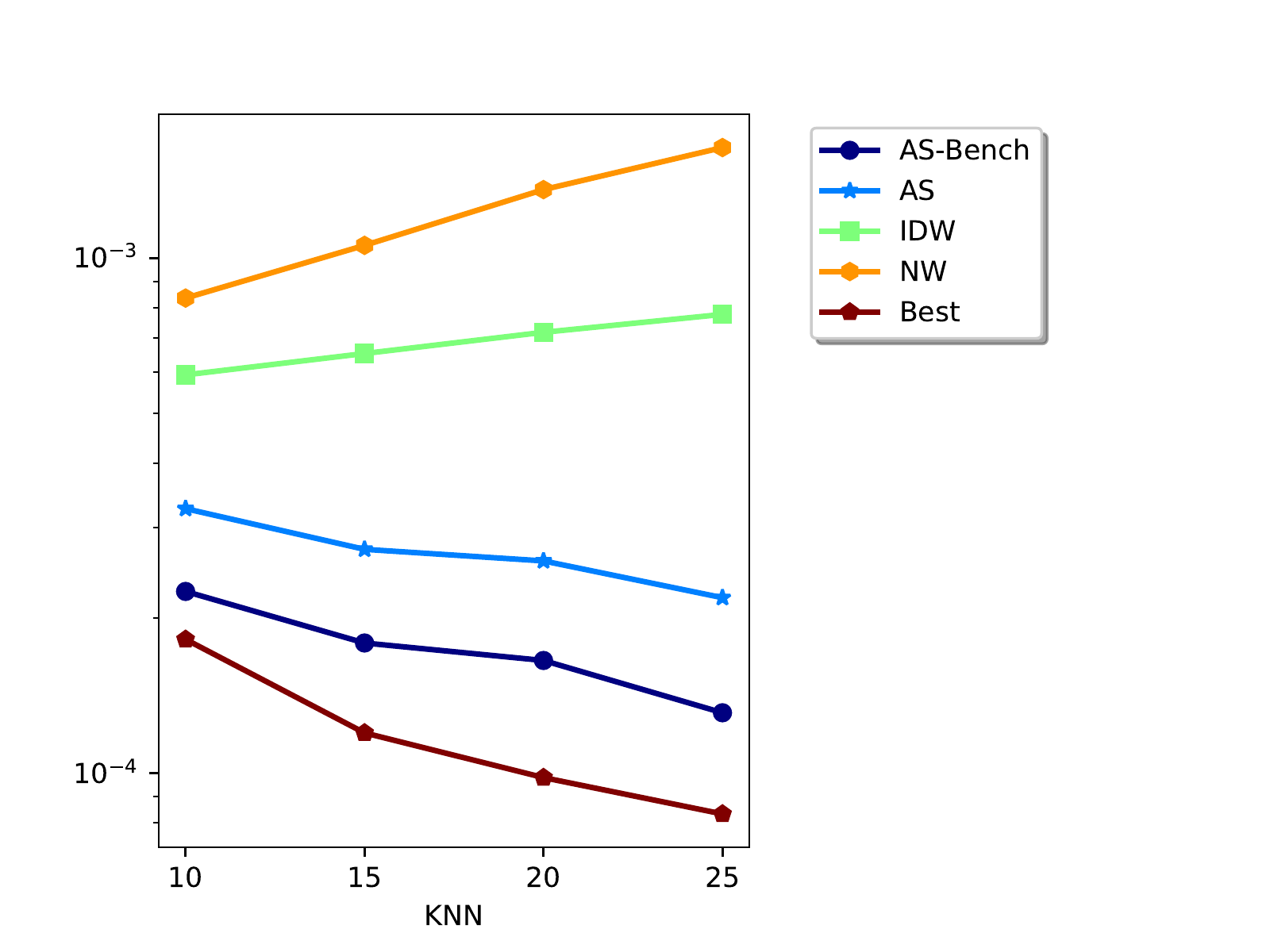}
	\caption{Average error as a function of the number $k$ of nearest neighbors.}
	\label{fig:error_curves_MOR}
\end{figure}

\begin{figure}[p]
	\centering
	\includegraphics[scale=0.8]{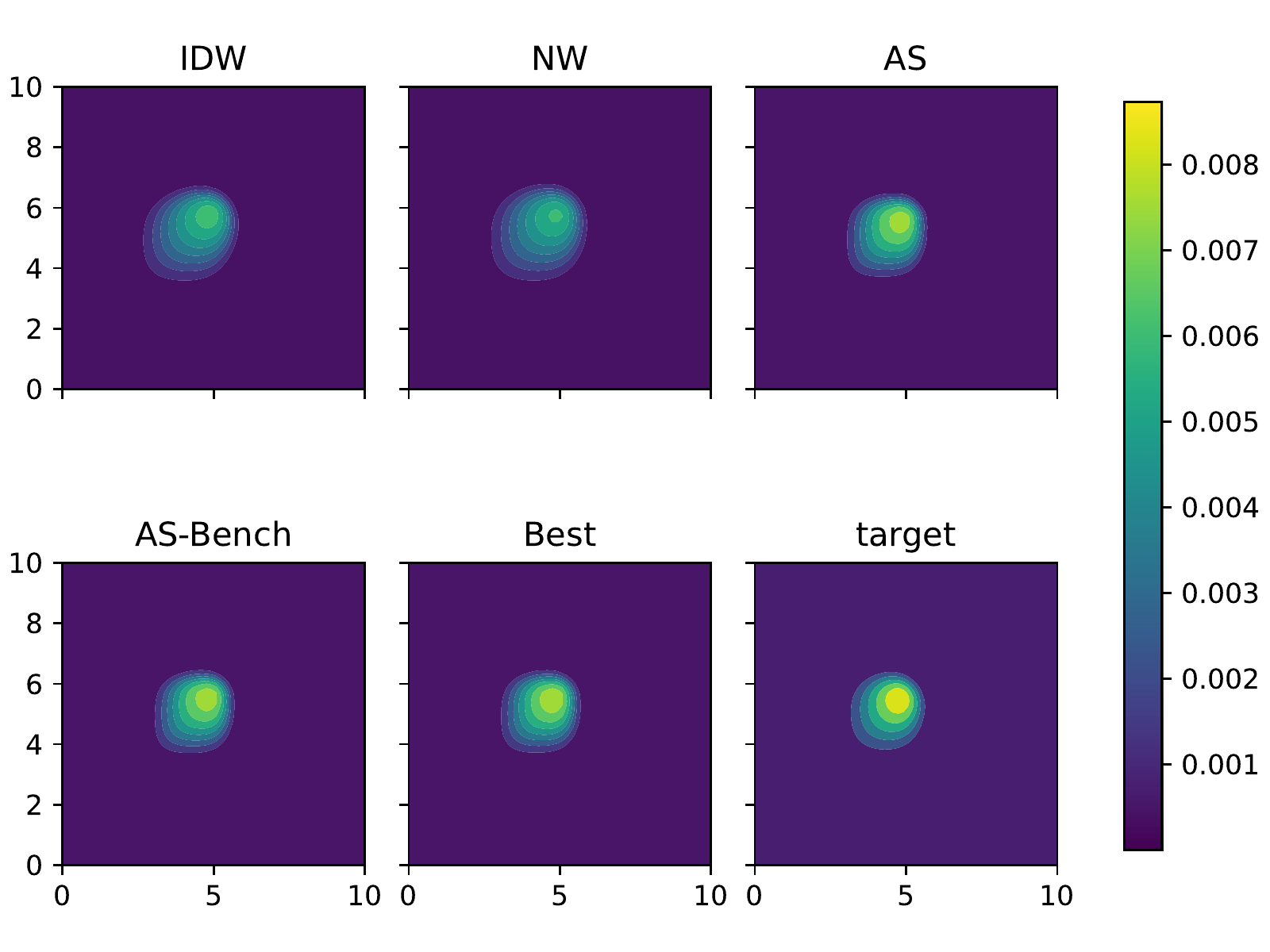}

	\caption{Target distribution and predictions made by different regression methods.}
	\label{fig:reconstruction_MOR}
\end{figure}

\paragraph{Test 1.1 -- distributions that belong to the input dataset:} Starting with a simple sanity check, we set $N=100$ and consider a target function $\ma$ which belongs to a database $\rY_N = \{ y_i\}_{i=1}^N$. The measures $y_i$ are solutions to Burgers' equation for parameters $x_i\in X$ picked randomly. For this example, we choose the measure $\alpha=y_1$ as the target. The minimum of the loss function in problem \eqref{eq:best-n-term-lambda-2} is thus equal to $0$ and is attained for $\weight_N(\ma) = e_1 = (1,0, 0\dots)$. 

We test the ability of \PG, \GAS~and \RGSP~to find the optimal solution of this problem, starting from a collection of uniform weights for the initial guess.
Figure~\ref{fig:PGD_in_loss_support_a} shows the evolution of the loss function across iterations for different values of the target sparsity degree $n$. Figure~\ref{fig:PGD_in_loss_support_b} displays the evolution of the support across iterations. The accuracy obtained by \GAS~and \PG~  is much better than that of \RGSP~ method. We observe that the learning rate of the algorithms is an important hyperparameter: as illustrated in Figure \ref{fig:PGD_in_lr}, a suitable learning rate may significantly improve convergence speed.  Figure~\ref{fig:PGD_in_prediction_weight} shows the results of the \GAS~algorithm with $lr=0.5$. As can be seen, we recover exactly the reference weight for our target function; the reconstructed prediction is identical to the target solution.

\paragraph{Test 1.2 -- distributions that belong to the convex hull of the dataset:} We use the same dataset ($N=100$) but consider a more complex target function that is the barycenter of $y_1$ and $y_3$ with reference weights $\weight_N = (0.3, 0, 0.7, 0, 0, \dots)$. Figure~\ref{fig:PGD_bary2_loss_support} shows that the \GAS~and \PG~algorithms with $n=10$ give better accuracy than the other strategies. Let us also note that the true support of the target function is equal to $2$, but that many other minimizers and near-minimizers to problem \eqref{eq:best-n-term-lambda-2} may exist with a larger support. This is especially true in a context where the dataset $\barsetsp_N$ contains redundant snapshots.
Figure~\ref{fig:PGD_bary2_prediction_weight} shows the eventual convergence of the projected gradient descent scheme to a suitable solution: the reconstruction is very close to the reference target. Please also note that although we do not recover exactly the reference weights, the coordinates at locations 1 and 3 are larger than the other ones.

\paragraph{Test 1.3 -- distributions that do not belong to the convex hull of the dataset:} Finally, we stick to the same simple dataset but consider a target measure $\ma$ that is a solution to Burgers' equation for some parameter vector $x$ but was not explicitly generated as a barycenter of distributions in the dataset $\barsetsp_N$. In this case, we  don't have access to a ground truth vector of interpolation weights. Figure shows that the cardinal of the optimal support for this target function is around $10$, with the \GAS~and \PG~methods run with $n=10$ providing the best reconstruction accuracy. As shown in Figure~\ref{fig:PGD_out_prediction_weight}, the reconstructed solution is not perfectly equal to the target function but provides a close approximation.

\begin{center}
	
	\begin{table}
		\qquad~
\begin{tabular}{lccccc }
\hline
Method & Acronym & Formula &  Type & Local & Nearest \\
& & & & metric & neighbors \\
\hline
Barycentric Greedy  & \BGA & \eqref{eq:greedy-induction} & Feasible & No & No\\
Algorithm & & & & & \\
First Nearest Neighbor & \NN & & Feasible & Yes & Yes \\
Inverse Distance Weighting & \IDW & \eqref{eq:K-IDW}, \eqref{eq:nadaraya} & Feasible & \underline{Yes}/No & Yes \\
Nadaraya-Watson & \NW & \eqref{eq:K-NS}, \eqref{eq:nadaraya} & Feasible & \underline{Yes}/No &Yes \\
Adaptive, Sparse & \AS & \eqref{eq:weights-EE-v2-full} & Feasible & Yes & \underline{Yes}/No \\
Adaptive, Sparse & \AS-Bench & \eqref{eq:weights-EE-v1} & Benchmark & No & No \\
Benchmark & & & & & \\
Best Barycenter & \Best & \eqref{eq:best-n-term-lambda-fx} & Benchmark & No & No \\
\end{tabular}
\caption{Summary of strategies tested for Model Order Reduction.}
\label{summary-strategies}
\end{table}
\end{center}

\subsection{Performance for Model Order Reduction}
In this section, we compare some of the above discussed strategies for structured prediction. The comparison is done for an example of model reduction of the Burgers' equation. For the sake of clarity, we summarize the strategies that we have tested and their main features in Table~\ref{summary-strategies}. In the ``Type'' column, we record whether the method is feasible in practice, or if it is a benchmark that relies on an oracle for the ground truth distribution $f(x)$ to produce some notion of optimal performance. We also recall if the precomputation of local Euclidean metrics or of nearest neighbors is required. When several options are available, we underline our choice in this series of experiments.

We work with a set of $N=100$ training snapshots and $M=100$ snapshots for validation. We fix the maximum sparsity degree to $n_{\max}=10$ so that all proposed methods compute Wasserstein barycenters that involve at most 10 training snapshots. The number of nearest neighbors is fixed to $k=10$. We consider $\sigma = \frac{1}{2}$  for the Gaussian kernel of the Nadaraya-Watson interpolation and $p=1$ for Inverse Distance Weighting method.

Figure~\ref{fig:boxplot_MOR} shows a boxplot of approximation errors on the validation set. The best barycenter method serves as a benchmark for the absolute optimal performance that one can obtain when approximating the validation set using Wasserstein barycenters of the training distributions. We see that, overall, feasible methods perform relatively well with accuracy that is degraded by less than one order of magnitude.

Among all feasible methods, the \AS~approach stands out as the one that provies the best accuracy. We can examine how much accuracy is lost in the Euclidean embedding step by comparing \AS~with \AS-Bench, which implements the same strategy but relies on an oracle to get access to the true Wasserstein distances between the ground truth solution $f(x)$ and the training distributions $f(x_i)$. We observe that we lose a factor of $2$ in terms of accuracy, which suggests that our local Euclidean metric is performing reasonably well but imperfectly -- as confirmed in Figure~\ref{fig:LE}.

Going further, an interesting observation that stems from Figure~\ref{fig:boxplot_MOR} is that the greedy approach \BGA{} performs similarly to the naive adaptive approach of selecting the first $k=10$ nearest neighbor of the target measure using the Local Euclidean Embedding, and computing the best weights with an oracle nearest neighbor projection. 
The same remark applies to the kernel-based methods \IDW{} and \NW. Since these two methods are based on a certain a priori assumption regarding the behavior of the barycentric weights, the fact that they are outperformed by the data-driven strategy reveals that the kernel heuristic is not optimal.

One can also compare the methods with respect to the computational effort that is required for the training phase. From this perspective, \BGA{} requires a lot of time to find the $n=10$ global basis functions (see Table \ref{table:run-time}). During the online inference step, the method is lightweight as we simply interpolate vectors of weights using a radial basis function and compute a single Wasserstein barycenter of size $n=10$. All the other strategies rely on the precomputation of local Euclidean metrics during the offline step, with a significantly shorter run time compared to \BGA{}. However, during the online inference step, we have to either do some interpolation or optimization to estimate the weights. We report in Table \ref{table:run-time} the obtained run times in the case of \AS{}. The time reported in the inference corresponds to one iteration of the algorithm. In our case, we did not optimize the number of steps and we performed 200 iterations for each sample.


\begin{table}
\centering
\begin{tabular}{lcc}
	\hline
	Method & Training time (in s.) & Predict time per sample (in s.)\\
	\hline 
	\BGA{} & $1.02.10^{7} (\approx 11.9 \text{ days })$  &  $6.02$ \\
	\AS{} & $64.78$ & $7.07$ 
\end{tabular}
\caption{Most relevant run time values.}
\label{table:run-time}
\end{table}

We end this section by illustrating the behavior of the methods when the number $k$ of nearest neighbors varies. Figure~\ref{fig:error_curves_MOR} shows the average error on the validation set as a function of $k$. We observe that as $k$ increases, the error decreases for \AS~type methods but that this is not the case for kernel-based methods. This ``instability'' is another indication that the heuristics underlying these methods are not totally optimal -- at least for a naive choice of kernel function.

Last but not least, Figure~\ref{fig:reconstruction_MOR} illustrates predictions made by the different regression methods for a target solution that corresponds to a parameter $x$ drawn at random.

\section{Conclusion}

This paper is a contribution towards efficient numerical methods for sparse approximations and structured prediction in $\Pr$ through Wasserstein barycenters. We have introduced the concept of best $n$-term barycenter, with tractable algorithms. We then explain in what sense this notion provides a benchmark of optimal performance for structured prediction with sparse barycenters. Since the best $n$-term barycenter cannot be computed without access to an oracle, we have introduced a feasible, fully adaptive and sparse interpolation method. This strategy generalizes classical reduction concepts such as Principal Component Analysis in vector spaces, or tangent PCA and barycentric greedy approaches in $\Pr$.  The behavior of all the proposed algorithms is overall superior to existing approaches in the simple numerical examples that we have considered. However, the limitations of the Wasserstein barycenteric construction also prevent us from applying this method convincingly on multi-modal distributions $y_i$ that may not be approached well using Wasserstein barycenters. Future works will focus on extending the present ideas to measures that do not have the same mass, and to notions of interpolation between distributions that provide stronger guarantees than the optimal transport metric on the preservation of the supports' topologies.

\appendix
\section{Proof of Lemma \ref{lemma:cont-weights}}
\label{sec:proof-lemma}
The proof of this result crucially relies on the following minimum theorem  from \cite{Clarke1975}. We provide a statement which is slightly adjusted to our current purposes.

\begin{theorem}[Theorem 2.1 of \cite{Clarke1975}]
\label{th:min-thm}
Let $\cV$ be a sequentially compact space, and let $h:\bR^N\times \cV \to \bR$ have the following properties:
\begin{itemize}
\item $h$ is lower semi-continuous in $(\omega, \nu)$.
\item $h$ is Lipschitz in $\omega$, uniformly for $\nu\in \cV$.
\item $\partial_\omega h(\omega,\nu)$ is lower semi-continuous in $(\omega, \nu)$.
\end{itemize}
Then, if we let $\bar h(\omega) = \min_{\nu\in \cV} h(\omega, \nu)$, we have that $\bar h$ is Lipschitz, and differentiable in $\bR^N$.
\end{theorem}

\begin{proof}[Proof of Lemma \ref{lemma:cont-weights}] We start by fixing the set $\barset_N$. The proof then consists in:
\begin{enumerate}
\item[i)] Proving that there exists a weakly sequentially compact subset $\cV\subset \cP_2(\Omega)$ such that for all $\weight_N\in \Sigma_N$,
\begin{equation}
\label{eq:min-bary-ball}
\bar L(\weight_N) = \min_{\mgI\in \cV} L(\weight_N, \mgI).
\end{equation}
In other words, we can minimise over $\cV$ instead over $\Pr$ in the definition of the barycenter.
\item[ii)] Verifying that the function $h = L$ satisfies the conditions of Theorem \ref{th:min-thm}. This way, by application of the theorem, we can conclude that $\bar L$ is Lipschitz, and differentiable in $\Sigma_N\subset \bR^N$.
\end{enumerate}

To prove i), we fix $\weight_N\in \Sigma_N$ and we consider the set $\cS=\cup_{i=1}^N \cB(\ma_i, \bar r)$ where $\cB(\ma_i, \bar r)$ is the ball of center $\ma_i$ and radius
$$
\bar r \coloneqq \max_{(\mu,\nu)\in \barset_N\times \barset_N} W_2(\mu,\nu).
$$
Since $\cS$ is bounded, there exists a ball $\cV$ that contains it, and for all $\weight_N\in \Sigma_N$,
$$
L(\weight_N, \nu) > \bar r ,\quad \forall \nu \not\in \cV
$$
whereas
$$
L(\weight_N, \nu) \leq \bar r ,\quad \forall \nu \in \cV.
$$
Therefore the infimum in the barycenter problem \eqref{eq:bary-proof-compact} is in $\cV$. We next prove that the infimizer is indeed a minimizer, namely that there exists a measure $\mgI_\infty$ in $\cV$ that minimizes $L(\weight_N, \mgI)$ over all $\mgI\in \cV$. For this, let $\mgI_n \in \cV$ be an inifimizing sequence. Since the ball $\cV$ is weak sequentially compact in $\cP_2(\Omega)$ (see \cite{YKW2021}), up to extracting a subsequence, there exists $\mgI_\infty\in \cV$ such that $\mgI_n  \rightharpoonup^* \mgI_\infty$. Finally, since $\nu \mapsto L(\weight_N, \nu)$ is lower-semi continuous with respect to the weak convergence, then $\underset{n\to \infty}{\lim\,\inf}\, L(\weight_N, \mgI_n)\geq L(\weight_N, \mgI_\infty)$, thus proving that $\mgI_\infty$ is a minimiser of \eqref{eq:min-bary-ball}.

To prove ii), it suffices to verify that the function $h=L$ satisfies the conditions of Theorem \ref{th:min-thm} so we can apply it to derive the desired continuity result.

Finally, weak sequential compactness of the set $\bary(\Sigma_N, \barset_N)$ follows from compactness of $\Sigma_N$ in $\bR^N$ and the continuity of the application $\weight_N \to \bary(\weight_N, \barset_N)$.
\end{proof}

\bibliographystyle{unsrt}
\bibliography{references}

\end{document}